\documentclass[11pt]{amsart}
\usepackage[latin1]{inputenc}
\usepackage{amsfonts,amssymb,amsmath,amsthm,enumerate,latexsym,color}
\usepackage{graphicx} %,multimedia,url}

\newtheorem{thm}{Thm}[section]

\newtheorem{rem}{Remark}[section]
\newtheorem{Step}{Step}[section]
\newtheorem{prop}{Proposition}[section]

\newcommand{\R}{\mathbb{R}}
\newcommand{\C}{\mathbb{C}}
\newcommand{\eps}{\varepsilon}
\newcommand{\p}{\partial}
\newcommand{\D}{\displaystyle}
\newcommand{\pmean}{\langle p\rangle}
\newcommand{\qmean}{\langle q\rangle}

\title[Opinion formation models with heterogeneous... ]{Opinion formation models with heterogeneous persuasion and zealotry}

\begin{document}

\author[M. P\'erez-Llanos, J.P. Pinasco, N.Saintier, and A. Silva]{Mayte P\'erez-Llanos,
Juan Pablo Pinasco, Nicolas Saintier, and Anal\'ia Silva}

\address{M. P\'erez-Llanos,
\hfill\break\indent IMAS UBA-CONICET  and  Departamento  de Matem{\'a}tica,
 \hfill\break\indent Facultad de Ciencias Exactas y Naturales, Universidad de Buenos Aires,
 \hfill\break\indent  Av Cantilo s/n, Ciudad Universitaria
 \hfill\break\indent (1428) Buenos Aires, Argentina.}
 \email{{\tt maytep@dm.uba.ar,}}

\address{ J.P. Pinasco,
\hfill\break\indent IMAS UBA-CONICET  and  Departamento  de Matem{\'a}tica,
 \hfill\break\indent Facultad de Ciencias Exactas y Naturales, Universidad de Buenos Aires,
 \hfill\break\indent  Av Cantilo s/n, Ciudad Universitaria
 \hfill\break\indent (1428) Buenos Aires, Argentina.}
 \email{{\tt  jpinasco@gmail.com}}

\address{ N. Saintier
\hfill\break\indent IMAS UBA-CONICET  and  Departamento  de Matem{\'a}tica,
 \hfill\break\indent Facultad de Ciencias Exactas y Naturales, Universidad de Buenos Aires,
 \hfill\break\indent  Av Cantilo s/n, Ciudad Universitaria
 \hfill\break\indent (1428) Buenos Aires, Argentina.}
 \email{{\tt  nsaintie@dm.uba.ar}}

\address{A. Silva
\hfill\break\indent Instituto de Matem\'atica Aplicada San Luis, IMASL.
 \hfill\break\indent Universidad Nacional de San Luis and CONICET.
 \hfill\break\indent
Ejercito de los Andes 950.
 \hfill\break\indent D5700HHW San Luis, Argentina.}
 \email{{\tt analia.silva82@gmail.com}}

%\thanks{Supported  }

\begin{abstract}
 In this work  an opinion formation model with heterogeneous
agents is proposed. Each agent is supposed to have different power of persuasion, and besides its
own level of zealotry, that is, an individual willingness to being convinced by other
agent. In addition, our model includes zealots or stubborn agents, agents that never change opinions.

We derive a Bolzmann-like equation for the distribution of agents on the space of
opinions, which is approximated by a transport equation with a nonlocal drift term.
 We study the long-time asymptotic behavior of  solutions, characterizing the limit distribution of agents,
which consists of the distribution of stubborn agents, plus a delta function at the
mean of their opinions, weighted by they power of persuasion.

  Moreover,   explicit bounds on the rate of convergence are given,  and the
time to convergence is shown to decrease when the number of stubborn agents increases. This is a
remarkable  fact observed in agent based simulations in different works.
\end{abstract}

\maketitle

\section{Introduction}

In recent years, opinion formation, as well as other sociological and economical
phenomena, have attracted  a considerable attention from physicists and
mathematicians, as it was realized that concepts from statistical mechanics could be
successfully applied to model them. We refer for instance to the papers by  S. Galam
 \cite{G1,GZ}, Sznajd and Sznajd-Veron \cite{S}, Deffuant et al. \cite{DNAW,DAWF}, and
Slanina \cite{Sla} among other works. From them, quickly emerged two very active new
fields, usually called {\it sociophysics} and {\it econophysics}, devoted to the
description of these phenomena from the physicists point of view. We underline some
recent books \cite{CCC,Gal,SC,Sla} for an overview and up-to-date references.

In the sociophysics community, a customary procedure for  modelling the formation of
opinions in a population consists in representing the opinion of an individual, with
respect to certain subject, by a real number. This number can vary in some discrete
set or in a fixed interval, say $[-1,1]$, meaning  $-1$ to be completely against the
subject. Individual changes of opinion are assumed to be a result of binary random
interactions between agents. The opinions $w$ and $w_*$ of two agents will turn  to
new opinions $w'$ and $w_*'$ as a consequence of the discussion enclosed by the two
agents, and also  by the influence  of external factors such as media or propaganda, and
spontaneous changes of mind. Denoting by $f(t,w)$ the proportion of agents in the
population with opinion $w$ at time $t$, it is possible to describe the time evolution
of $f(t,.)$ with a Boltzmann-like equation, whose collision part reflects the dynamics
in the changes of opinion due to encounters. Thus, the long-time asymptotic behavior
of $f(t,.)$ can be analyzed theoretically and/or numerically.

This procedure, which is by far not the only way of modelling opinion formation, is
strongly inspired by the kinetic theory of rarefied gases and granular flows. The
recent advances in the mathematical foundations of kinetic theory (see
\cite{vil1,vil2}), motivated several mathematicians to perform a rigorous study of
this kind of problems, using tools from partial differential equations, optimal
transport, game theory and stochastic processes.  This approach has been successfully
implemented by Bellomo, Ben-Naim, Pareschi, Toscani and their collaborators in a wide variety of
settings, we refer the interested reader to their works \cite{ABD,Bell,Bell2,BenNaim1,BenNaim2,PT} and
the surveys in Ref. \cite{NPT} for further details.

\medskip

In this work we introduce  a continuous  model of opinion formation, where agents
opinions are real numbers in $[-1,1]$, and agents change their opinions  through
binary interactions  as mentioned before.
Most of the agents are assumed to have some propensity to reach an agreement, the
so-called compromise hypothesis and hence after each interaction they tend to get
closer positions.  However, we introduce  a high degree of heterogeneity among the
agents,  and each agent $i$ has a priori two individual characteristics:
\begin{itemize}
\item some power of persuasion,  represented by a probability $p_i \in[0,1]$ that the
agent
    will convince the other agent involved in the interaction, and

\item some willingness to change his/her own opinion, represented by a probability
    $q_i\in[0,1]$  that the agent is persuaded.
\end{itemize}

Observe that the assumption that $q$ could be zero introduces {\it zealots} or
stubborn individuals, i.e., agents who have strong opinions and they are not affected
by other agents' opinions, not changing their mind after interactions. The presence of stubborn agents was studied mainly
in discrete  models of opinion
dynamics, related to consensus formation, game theory models, and diffusion of
innovations, among other applications, see  \cite{DL,lipow,Mob,WVC,XEK,xie,Ozg}.

In these works it is shown, mainly through simulations, how the stubborn agents affect the process of consensus formation,
specially the kind of expected equilibria that could arise due to their influence, and
the time to convergence. Let us remark that in \cite{Ozg} and related works, several
results were proved theoretically using probabilistic arguments. Also, a striking
fact was observed in the simulations: the time to convergence decreases when the
number of stubborn agents increases.

Much fewer in number are the works considering continuous opinion models with zealots
or persuasion, see for instance \cite{BrT,autom}. Let us note that the presence of
leaders and followers as in During et al. \cite{DM} has a somewhat similar dynamics
when a leader and a follower interact, since only the follower updates its opinion.
However, the interactions among leaders are allowed, and hence they can change their
opinions.

In the work \cite{DW}, the authors presented a related model
including   an additional
variable representing the assertiveness level of agents,  similar to our variable $p$. Now,  the assertiveness evolves in time, and a
Matthew effect, or Rich-Gets-Richer dynamics, is proposed, where after collisions the
agent with higher  (respectively, lower) assertiveness increases (resp., decreases) its value.   On the other
hand, the agents in their
model are always receptive to other agents opinions, and no zealots are present.

Recently, in \cite{WVaz}, each agent has  some parameter
$k$, which is a mix between zealotry and assertiveness, which also evolve in time, and zealots can appear dynamically in the model.

\bigskip

Our aim in this work is to rigorously show that the long time behavior of the agent
based model can be described with a  Boltzmann-like equation satisfied by the
distribution of agents $f(t,.)$, and it is properly
    approximated by a non-linear non-local transport equation, which is well-posed for
    measure-valued functions.
    We then establish the convergence of the solution to some
    limit density, with explicit bounds on the time of convergence.

    Essentially, the limit reveals that the part of the population composed by
    individuals  who are willing to change their opinion tends to share the same opinion.
    Furthermore, we find out that this limit opinion is precisely the mean opinion of the
    stubborn individuals, those who always keep their own opinion ($q=0$),
    weighted by their power of persuasion.

  Moreover, the bounds for the rates of convergence  point out that, the greater the number of stubborn
    individuals is, the faster the system reaches the stationary state. This  fact has been observed in our simulations and also in related discrete opinion
    models, see for example \cite{lipow,Mo1,Mob,xie},  exhibiting that the asymptotic distribution of opinions in the
    population is completely determined by the stubborn individuals, although their
     influence take a long time to be observed when there
    are just a few of them.

\medskip

\subsection{Notations and definitions} We  denote $K=[-1,1]\times [0,1]\times [0,1]$ and a generic
point of $K$ as $\varpi=(w,p,q)$. Let $P(K)$ be the convex set of probability measures
on $K$. Given $f\in P(K)$, we  write the integral of a function $\phi$ against $f$
as $\int_K\phi(\varpi)\,df(\varpi)$ or as $\int_K\phi(\varpi)f(\varpi)\,d\varpi$. The
expression $f(\varpi)\,d\varpi$ is merely a notation, we are not assuming that $f$ has
a density necessarily. By $f(w)dw$ we understand the marginal of $f$ with respect to the
first variable $w$, namely
$$ \int_{A\times [0.1]\times [0.1]} f(\varpi)\,d\varpi
=\int_A f(w)dw \qquad \text{for any $A\subset [-1,1]$ Borel.} $$ Here again $f(w)dw$
is merely a notation, we are not assuming in general that this measure has a density.

\medskip
$M(K)$ stands for the space of finite measures on $K$,  $M_+(K)$ denotes the cone of
nonnegative measures and  $P(K)$ the convex cone of probability measures.
These sets are endowed  with the total variation norm, namely
\begin{equation}\label{norm}\|f\|
=\sup\left\{\int_K \phi\,df:\, \phi\in C(K) \mbox{ such that }\|\phi\|_\infty\le
1\right\}.
\end{equation}
Let us remark  that  $M(K)$ becomes a Banach space with this norm.

\medskip

Later on, we will need to endow the set $P(K)$ with the weak convergence of the measure
topology. It will be convenient then, to recall from \cite{V} that the  Wasserstein distances $W_p$,  with $p\ge 1$,
between two probability measures $\mu,\nu$ are defined as
\begin{equation}\label{DefWasserstein1}
  W_p(\mu,\nu)
  =\left(\inf_{\alpha\in \Gamma(\mu,\nu)}\int_{K\times K} |x-y|^p d\alpha(x,y)\right)^{1/p},
\quad p\geq 1,
\end{equation}
being $\Gamma(\mu,\nu)$ the collection of all measures on $K\times K$ with marginal
measures $\nu$ and $\mu$ on the first and second factor, respectively. When $p=1$ the  Kantorovich and Rubinstein Theorem provides a dual representation of $W_1$, namely
\begin{equation}\label{DefWasserstein2}
W_1( \nu,\mu) = \sup\left\{\int_K\varphi \,d(\mu-\nu):\quad \varphi\; \mbox{ is 1-
Lipschitz}\right\}.
\end{equation}

\medskip

\medskip

\subsection{Organization of the paper}

  \medskip

The paper is organized as follows.

 Section \S 2 contains a detailed description of the rules
governing the updates of the individuals opinions during encounters,
determining a Boltzmann-like equation satisfied by the agent
distribution  $f(t,.)$. We  introduce in addition the so-called grazing limit
that yields a Fokker-Planck equation,  modelling the long-time
asymptotic behavior of the density, when the interactions among the
agents produce very tiny changes in their opinions, namely when the
parameters $\sigma,\gamma\to 0$. This idea of studying Boltzmann-like
equation in the limit of small changes in each interaction comes
from the literature about the Boltzmann equation (see e.g.
\cite{DL,D,DR,V2,V3} and references therein) and was first applied
in the context of opinion formation model by Ben-Naim, Krapivsky and Redner \cite{BenNaim1}, and by Toscani \cite{T}.

In
Section \S 3  we derive
to the analysis of the asymptotic behavior as $t\to +\infty$ of the equation arising
when in the grazing limit $\frac{\sigma^2}{\gamma}\to 0$, namely, when the transport term dominates the diffusive term,
$$\p_tf(w,p,q) + \p_w((m_t-w)q \pmean f(w,p,q))=0,
$$
where $\pmean $ is the mean value of  the persuasion power $p$, which remains constant
in time, and  $m(t)=\int_K \frac{p}{\pmean}w\,df_t(\varpi)$ is the weighted mean
opinion.

We characterize the limit
distribution of agents, which consists of the original distribution of stubborn agents, plus a
delta function at the mean of their opinions, weighted by they power of persuasion. We
determine explicit bounds on the rate of convergence and  show that the time to
convergence decreases as the number of stubborn agents increases.

Finally, some  computational experiments illustrating our
theoretical results are included in Section \S 4. We perform an
agent based simulation of the dynamics of our problem without noise, in order to verify that  its stationary distribution coincide with the
theoretical one.

In a Supplement we include for a sake of completeness the proof of  the existence and uniqueness of solutions to
the Boltzmann equation introduced in section 2, using the ideas in
Chapter 6 of the book of Cercignani, Illner and Pulvirenti
\cite{CI}.
 For the reader's convenience, also we provide in this Supplement a detailed proof of the
approximation of the Boltzmann-like equation by a diffusion-transport equation, via the grazing limit.
This proof is based mainly on Toscani \cite{T}.

\section{Description of the model}

\subsection{Microscopic interaction rules}

Let us introduce our model of opinion formation. We consider a population composed by
$N$ agents. The opinion of an agent with respect to certain matter is represented by a
real number $w\in [-1,1]$ (meaning $-1$ being completely in disagreement with the
subject and $1$ in complete agreement). In addition, we take into account the ability
(or difficulty) of an individual to persuade another agent, as well as his/her
reticence (or facility)  to change his/her opinion. We denote by $p\in[0,1]$ the
probability of the agent to convince the opponent and by $q\in[0,1]$ the probability
that the agent is persuaded to change his/her own opinion. Each agent is thus
characterized by three parameters $(w,p,q)$.

Agents' parameters $(w,p,q)$ could be modified during binary encounters. For
simplicity, in this work the parameters $(p; q)$ are assumed to be fixed and to remain
unchanged in time, although there exist models where the agents' persuasion also
evolve, as in \cite{BrT,Vaz}.

We now describe the up-dating rules of the opinions. Consider two interacting agents
with parameters $(w,p,q)$ and $(w_*,p_*,q_*)$  before the encounter. Denote by
$(w',p',q')$ and $(w'_*,p'_*,q'_*)$ the new values for the parameters after the
interaction, respectively. As we mentioned before, the parameters $(p,q)$  will remain
unchanged: $p'=p$, $q'=q$, $p_*'=p_*$, $q'_*=q_*$. Regarding the up-dating of the
opinion, we propose the following rule:
\begin{equation} \label{Regla}
\begin{array}{rl}
w' & =  w+\gamma qp_*(w_*-w)+\eta q D(|w|), \\
w_*'& =  w_*+ \gamma  pq_*(w-w_*)+\eta_* q_*D(|w_*|).
\end{array}
\end{equation}
Observe that the change of opinion  $w'-w$ is the sum of two parts. On the one hand,
the term $\gamma qp_*(w_*-w)$  reflects the idea that the agents tend to reach a
compromise. This tendency is directly proportional to both his/her  willingness of
changing his/her own opinion, $q$ and also the power of persuasion  of the opponent,
$p_*$. Here  $\gamma$ is a given real number  in $(0,1/2)$ modelling the strength of
the interaction.

On the other hand,  the term $\eta q D(|w|)$ represents the inclination of an agent to
change his/her opinion due to random external or internal factors. Obviously, this
term is proportional to the facility $q$ of the agent to modify his/her opinion.  By
$\eta$ and $\eta_*$ we denote two independent and identically distributed random
variables, with null expected value and variance $\sigma^2$. More precisely, we will
write $\eta=\sigma Y$, being $Y$ a symmetric random variable such that $E[Y]=0$,
$Var[Y]=1$ and $E[|Y|^3]<\infty$, and the same is assumed  for $\eta_*$. The function
$D(|w|)\in [0,1]$ is supposed to be non-increasing in $|w|$. Some typical examples are
$1-w^2$, $1-|w|$ and $\sqrt{1-w^2}$. Notice that in these examples $D(\pm 1)=0$. This
is in accordance with the fact  that  the more extreme an opinion is, the more
difficult to be changed.

\subsection{Macroscopic kinetic model: Boltzman equation}

Let $ f(t,\varpi)$ be the distribution  agents with opinion $\varpi$ at time  $t\geq0$, hence
$f(t,\cdot)$ is a probability measure on $K$. We  usually denote this measure as
$df_t$ or $f_t(\varpi)d\varpi$ bearing in mind that $f_t$ may not necessarily be absolutely
continuous with respect to Lebesgue measure. In fact, $f_t$ could be a Dirac measure.

 In case of  binary interactions the time evolution of the density $f$ is  a balance between   gain and loss of opinion
terms through an integro-differential equation of Boltzmann type:
\begin{equation}\label{EqBoltzmann}
\begin{array}{ll}
\D\frac{d}{dt}\D \int_K  \phi(\varpi) \, df_t(\varpi)
  \\ \qquad=\D\int_{B^2}\int_{K^2}\D\beta_{(w,w_*)\to(w',w_*')}( \phi(\varpi')-\phi(\varpi)) \, df_t(\varpi)df_t(\varpi_*)
    d\eta d\eta_*,
\end{array}
\end{equation}
for any  $\phi\in C^\infty(K)$, see  \cite{CI}. The kernel $\beta$ is related to the
transition rate  and takes into account the external events acting on the opinion. For
simplicity, we can take
$$ \beta_{(w,w_*)\to(w',w_*')}
  =\theta(\eta)\theta(\eta_*)\chi_{|w'|\leq 1}\chi_{|w_*'|\leq 1}, $$
where by $\chi_A$ we understand the indicator function of the set $A$ and $\theta$ is a
symmetric probability density  with zero mean and variance $\sigma^2$, characterizing
the diffusion of information. To avoid the dependence of $\beta$ on the probabilities
$w$, $w_*$ through the indicator function, we can ensure the boundedness of $|w'|$ and
$|w'_*|$,  assuming that the support of the random variables $\eta,\eta_*$ is
conveniently delimited. This reckons on the choice of the function $D$; for instance,
if $D(|w|)=1-|w|$ it suffices to take  $B=(-(1-\gamma), 1-\gamma)$ to obtain $|w'|\leq
1,\;\;|w_*|\leq 1$, while if $D(|w|)=1-w^2$, it is enough to have $|\eta|\le
\frac{1-\gamma}{2}$ since then $|\eta|\le \frac{1-\gamma}{1+|w|}$ (see \cite{T}). With
these choices, equation (\ref{EqBoltzmann}) corresponds to a classical Boltzmann
equation
\begin{equation}\label{BoltzDebil2}
 \D\frac{d}{dt} \int_K  \phi(\varpi) \,df_t(\varpi)
  =\D\int_B \int_{K^2} ( \phi(\varpi')-\phi(\varpi)) \,df_t(\varpi) df_t(\varpi_*) d\theta(\eta),
\end{equation}
for any  $\phi\in C^\infty(K)$.

Taking  $\phi(\varpi)=p$ and  $\phi(\varpi)=q$ as test functions in
\eqref{BoltzDebil2}  it is easy to see that the average of the persuasion ability,
$\pmean$,  and of the zealotry, $\qmean$, respectively, are constant in time. We
assume that $\qmean >0$, otherwise no opinion will change.

\medskip

\medskip

Our first concern is to show the existence of a solution to
\eqref{BoltzDebil2}. This is the purpose of the following Theorem.
The proof follows classical ideas and is detailed thereafter in the Supplement for the
reader's convenience.

\begin{thm} \label{ThmExistenceBoltzmann}
Given $f_0\in P(K)$, there exists a unique $f\in C^1([0,+\infty),P(K))$, where $P(K)$
is endowed with the total variation norm $(\ref{norm})$, such that
\begin{equation}\label{BoltzDebil}
\begin{array}{rl}
& \displaystyle \int_K \phi(\varpi)df_t(\varpi)  \quad \\
 & \quad =  \displaystyle \int_K \phi(\varpi)df_0(\varpi)
 + \int_0^t \int_{K^2\times B} (\phi(\varpi')-\phi(\varpi))\,df_s(\varpi)df_s(\varpi_*)d\theta(\eta)ds,
\end{array}
 \end{equation}
for any $\phi \in C(K)$.
\end{thm}

\subsection{Grazing Limit}

Given some initial condition $f_0\in P(K)$, consider the function $f$ solution to the
Boltzmann-like equation \eqref{BoltzDebil2} given by Theorem
\ref{ThmExistenceBoltzmann}. We will prove that, after an appropriate time rescaling,
the asymptotic behavior of $f(t)$ as $t\to +\infty$ is well-described when
$\gamma,\sigma\to 0$ by the solution $g\in C([0,+\infty),P(K))$ of some diffusion
equation, whose form depends on  the limit of the quotient $\frac{\sigma^2}{\gamma}$.
Namely,  it reckons on  the balance between the diffusion  strength, represented by
$\sigma$ and the tendency to an agreement, measured by the parameter  $\gamma$.
Indeed, we will see that in case they are proportional, i.e,  $\sigma^2=\gamma\lambda$
for some $\lambda>0$, then $g$ satisfies
\begin{eqnarray}\label{Fokker-planck}
\begin{array}{rl}
  \frac{\displaystyle d}{\displaystyle d\tau} \displaystyle \int_K  \phi(\varpi)g_\tau(\varpi) =
 & \displaystyle \int_K \Big((m(\tau)-w) \pmean q\Big)\partial_w\phi(\varpi)\,dg_\tau(\varpi)      \\
 & \hspace{0.4cm}
   +\frac{\lambda}{2} \displaystyle \int_K q^2 D^2(|w|)\p_{ww}\phi(\varpi)\,dg_\tau(\varpi),
\end{array}
\end{eqnarray}
for any $\phi\in C^\infty(K)$. Here $m(\tau)=\int_K
\frac{p}{\pmean}w\,dg_\tau(\varpi)$, where $\pmean$ is the mean value of $p$, which
remains constant in time. In other words, $m$ is the mean opinion weighted by the
normalized  power of persuasion.

Notice that \eqref{Fokker-planck} is the weak form of the  Fokker-Planck equation
\begin{equation}\label{FPClassic}
 \p_\tau g + \p_w\Big((m(t)-w)q\pmean g\Big)
 = \frac{\lambda}{2}\p_{ww}\Big(q^2D(|w|)^2 g \Big),
\end{equation}
subject to the following boundary conditions satisfied for any $\tau>0$:
\begin{equation}\label{condi1}
(m(\tau)-w)\pmean qg_\tau(\varpi) - \frac{ \lambda}{2}\p_w
\Big(q^2D^2(|w|)g_\tau(\varpi)\Big) =0, \qquad w=\pm 1, (p,q)\in [0,1]^2
\end{equation}
\begin{equation}\label{condi2}
D(|w|)^2 \int_0^1\int_0^1 q^2g(\tau,w)\,dpdq=0, \qquad w=\pm 1.
\end{equation}
These  conditions are the result of integrating by parts assuming that $g_\tau$ is
smooth. In a wide choice of noise terms $D(\pm1)=0$ (e.g. if
$D(|w|)=1-w^2$ or $D(|w|)=1-|w|$), thus \eqref{condi2} holds  straightforward  and
\eqref{condi1} simplifies into
\begin{equation}\label{condi1bis}
 (m(\tau)-w)\pmean qg_\tau(\varpi) = 0, \qquad w=\pm 1, p\in [0,1], q\in [0,1].
\end{equation}

The left-hand side of \eqref{FPClassic} corresponds to a transport equation describing
the tendency to agreement in the interacting rules.  It amounts to a transport towards
the mean opinion $m$ with a velocity being proportional to $q\pmean $, the product between
the tendency of an agent to change his opinion and the mean power of persuasion. The
right-hand side of \eqref{FPClassic} is a diffusion term representing the possibility
for an agent of changing his opinion under the influence of random external factors.

\medskip

Notice that the limit equation \eqref{FPClassic} has  both a diffusion and a transport
term according to the assumption $\frac{\sigma^2}{\gamma}\to\lambda>0$. If we suppose
instead that $\frac{\sigma^2}{\gamma}\to 0$ or that $\frac{\sigma^2}{\gamma}\to
+\infty$, the limit equation has only the transport term or the diffusion term,
respectively.

Namely, if  $\frac{\sigma^2}{\gamma}\to 0$, the limit equation turns out to be
$$   \int_K  \phi \,dg_\tau = \int_K \phi \,df_0
    + \int_0^\tau \int_K (m(\tau)-w)\pmean q\phi_w(\varpi) \, dg_s(\varpi) ds,  $$
    which is the weak formulation to the transport equation
$$
\p_tf(w,p,q) + \p_w((m_t-w)q \pmean f(w,p,q))=0.
$$

We are interested here in this case, and the above mentioned facts regarding the grazing limit are summarized in  the Supplement.
We provide a full detailed proof based on the arguments in \cite{T}.

\section{Asymptotic behavior of the Fokker-Planck equation without noise}

This section is concerned with   the asymptotic behavior as $t\to +\infty$ of  solutions to the
Fokker-Planck equation
 \begin{equation}\label{FPSinRuido}\p_tf(w,p,q) + \p_w((m_t-w)q \pmean f(w,p,q))=0,
\end{equation}
or its weak form \eqref{FPSinRuidoDebil}. This equation arises when in the grazing limit dominates the transport term, namely  $\frac{\sigma^2}{\gamma}\to 0$, see Theorem~\ref{ThmGrazingLimit}.

Recall that $\pmean $ is the mean value of  the persuasion power $p$, which remains
constant in time and that $m(t)=\int_K \frac{p}{\pmean}w\,df_t(\varpi)$ is the
weighted mean opinion.

The following observation ensures the uniqueness of solutions to \eqref{FPSinRuidoDebil}.
\begin{rem} \label{RemCCR}
Given $f_0\in P(K)$ and $f\in C([0,+\infty),P(K))$, $f(0)=f_0$,  it is easily seen
that the vector-field
$$E(t,\varpi):=v[f_t](\varpi)
:=\Big(\Big[\int\frac{pw}{\pmean}\,df_t(\varpi)-w\Big]q\pmean,0,0\Big) = ((m(t)-w)q
\pmean,0,0),$$ where $\pmean = \int_Kp\,df_0(\varpi)$, satisfies the following:
\begin{enumerate}
\item $E$ is continuous in $(t,\varpi)$, \item $|E(t,\varpi)|\le C$ for any
$(t,\varpi)$, \item $|E(t,\varpi)-E(t,\varpi')|\le C|\varpi-\varpi'|$ for any
$t,\varpi,\varpi'$.
\end{enumerate}
Moreover if  $g\in C([0,+\infty),P(K))$, $g(0)=f_0$, then
$$ \max_{\varpi\in K} |v[f_t](\varpi)-v[g_t](\varpi)|\le CW_1(f_t,g_t), $$
for any $t\ge 0$.

Invoking the theory developed in \cite{CCR} by Ca\~nizo, Carrillo and Rosado, we can
ensure that the equation $\p_tf+\text{div}(v[f_t](\varpi)f_t)=0$,  which is exactly
\eqref{FPSinRuido}, has a unique solution in
$C([0,+\infty),P(K))$ with initial condition $f_0$.
\end{rem}

The long time behaviour of the solution  will be accomplished by rewriting equation (\ref{FPSinRuido}) in a simpler
form due to Li and Toscani \cite{LT}. To apply this idea we need to bear in mind some facts
about the generalized inverse of the cumulative distribution function of a probability
measure. Only measures supported in $[-1,1]$ will be considered,  since this is the
case of interest in this paper, see the next subsection.

\subsection{A change of variable}

Let $f\in P([-1,1])$. The cumulative distribution function (cdf) $F:\R\to [0,1]$ of
$f$ is defined as $F(x)=f((-\infty,x])$. Notice that $F$ is non-decreasing and
right-continuous with left limit.

The generalized inverse of $F$ is defined as $F^{-1}:[0,1]\to [-1,1]$
\begin{equation}\label{DefGenInverse}
F^{-1}(\rho)= \inf\, \{ x\in [-1,1] \mbox{ s.t. } F(x)\ge \rho \}.
\end{equation}
Observe that $F^{-1}$ is non-decreasing, left-continuous with right limit in $(0,1]$
and
\begin{equation}\label{PropGenInverse1}
 [F^{-1}(0^+),F^{-1}(1)] \supset \text{supp}\,f.
\end{equation}
Furthermore, for any $x\in [-1,1]$ and any $\rho\in [0,1]$ the following inequalities
hold:
\begin{equation}\label{PropGenInverse2}
  \text{ If } F(x)>0 \text{ then } F^{-1}(F(x))\le x  \text{ while } F(F^{-1}(\rho))\ge \rho.
\end{equation}
See the note of Embrechts and Hofert \cite{EH} for the above (and further) properties of
$F^{-1}$.

%\begin{equation}\label{PropGenInverse3}
% F(x)\ge p>0 \Rightarrow x\ge F^{-1}(p),\qquad
% x\ge F^{-1}(p) \Rightarrow F(x)\ge p
%\end{equation}
%\begin{equation}\label{PropGenInverse4}
% F(x)<p \Rightarrow x\le F^{-1}(p)
% \end{equation}

The use of the generalized inverse enables us to  rewrite an equation like
\eqref{FPSinRuido} in terms of the generalized inverse of the cdf of $f_t$, and the
resulting equation is usually much simpler. More precisely, consider  $f\in
C([0,\infty);P([-1,1]))$ and let $F_t$ be the cdf of $f_t$ and $X_t=F_t^{-1}$ its
generalized inverse. Then, it can be proved that
\begin{equation}\label{PropGenInverse5}
 \int_0^1 \phi(X_t(r))\,dr = \int_{-1}^{1} \phi(w)\,df_t(w),
 \end{equation}
for any $\phi$ integrable (to prove this identity it suffices to check the formula for
$\phi$ of the form $1_{(-\infty,a]}$, $a\in\R$). This change of variables formula is the key of the next result.

\begin{prop} \label{EquGenInv}
Let $v:[0,+\infty)\times [-1,1]\to \R$ be continuous and globally Lipschitz with respect to the
second variable. Then $f \in C([0,+\infty),P([-1,1]))$ is a weak solution of
\begin{equation}\label{GenInvEqu1}
  \p_tf_t + \p_x(v(t,x)f_t) = 0,
\end{equation}
in the sense that for any $\phi\in C^\infty([-1,1])$ and any $t>0$,
\begin{equation}\label{GenInvEqu1Weak}
  \int_{-1}^1 \phi(x)\,df_t(x) =  \int_{-1}^1 \phi(x)\,df_0(x)
+ \int_0^t \int_{-1}^1 \phi'(x)v(s,x)\,df_s(x)ds,
\end{equation}
if and only if for any $r\in (0,1]$, $X_t(r)$ is a solution of
\begin{equation}\label{GenInvEqu2}
  \p_tX_t(r) = v(t,X_t(r)).
\end{equation}
Here $X_0$ is the generalized inverse of $F_0$ (the cdf of $f_0$).
\end{prop}

The proof can  be found essentially in Theorem 3.1 of Ref. [1]. However, we rewrite it
here under the point of view of the ordinary  equation for the flux
(\ref{GenInvEqu2}).

\begin{proof}
Assume that $X_t$ satisfies \eqref{GenInvEqu2}. Thanks to  \eqref{PropGenInverse5},
for any smooth $\phi$ we have that
\begin{eqnarray*}
\frac{d}{dt}\int_{-1}^1 \phi(x)\,df_t(x) & = & \frac{d}{dt}\int_0^1 \phi(X_t(r))\,dr
= \int_0^1 \phi'(X_t(r))v(t,X_t(r))\,dr  \\
& = & \int_{-1}^1 \phi'(x)v(t,x)\,df_t(x),
\end{eqnarray*}
which easily implies \eqref{GenInvEqu1Weak}.

Reciprocally, suppose that $f$ solves \eqref{GenInvEqu1}. By Fubini's theorem,
\begin{eqnarray*}
 \int_{-1}^1 \phi(x) F_t(x)dx
 & = & \int_{-1}^1 \phi(x) \int 1_{(-\infty,x]}(y)\,df_t(y) dx  \\
& = & \int_{-1}^1 \Big( \int_y^{+\infty} \phi(x)\,dx \Big) \,df_t(y).
\end{eqnarray*}
Differentiating with respect to time and taking into account  \eqref{GenInvEqu1} yield
$$ \frac{d}{dt}   \int_{-1}^1 \phi(x) F_t(x)dx = -\int_{-1}^1 \phi(y) v(t,y)\,df_t(y). $$
Moreover, $\p_x F_t=f_t$ in the distributional sense. Thus $F$ is a weak solution of
the transport equation
\begin{equation}\label{EquF}
\p_t  F_t + v(t,x) \p_xF_t = 0.
\end{equation}
Let $\phi_t(x)$ be the flow of $v$, i.e. the solution to
$\p_t\phi_t(x)=v(t,\phi_t(x))$, starting at $\phi_0(x)=x$. Then, as usual $F$ is
determined by the relation $F_t(\phi_t(x)) = F_0(x)$. It is now simple to conclude that
\eqref{GenInvEqu2} holds i.e. that $X_t(r)=\phi_t(X_0(r))$.

First  $F_t(\phi_t(X_0(r)))=F_0(X_0(r))$ which is greater than $r$ by
\eqref{PropGenInverse2}. Hence, for any $t$,  $\phi_t(X_0(r)) \ge X_t(r)$.

Conversely, for any $x<X_0(r)$ we have $F_0(x)<r$ so that $F_t(\phi_t(x))<r$ and then
$\phi_t(x)<X_t(r)$ by definition of $X_t(r)$. Letting $x\to X_0(r)$ we obtain
$\phi_t(X_0(r))\le X_t(r)$.

The proof is finished.
\end{proof}

\subsection{Conditional distributions}

 Another useful tool to achieve the asymptotic analysis is the concept of conditional distribution.

Let  $X,Y$ be two random variables defined over the same probability space with values
in $\R^d$ and $\R^k$, respectively,  and denote by $P_X$ and $P_{(X,Y)}$ the
distributions of $X$ and $(X,Y)$. Then there exists a map $\nu:(x,B) \in \R^d\times
\mathcal{F}(\R^k) \to \nu(x,B)\in [0,1]$ (where $\mathcal{F}(\R^k) $ is the Borel
$\sigma$-field) such that:
\begin{itemize}
\item[i)] $\nu(x,.)\in P(\R^k)$ for any $x\in \R^d$, \item [ii)] $\nu(.,B)$ is
measurable for any $B\subset \R^k$ Borel, \item[ iii)] $P(X\in A;Y\in B) =  \int_A
\nu(x,B)\,dP_X(x)$ for any $A\subset \R^d$, $B\subset \R^k$ Borel.
\end{itemize}
We write $\nu(x,B)=P(Y\in B|X=x)$.

The following Fubini formula holds: for any $\phi:X\times Y\to \R$
$P_{(X,Y)}$-integrable, the function $x\to \int_Y \phi(x,y)\, P(dy|X=x)$ is measurable
and
\begin{equation}\label{FubiniCondDistrib}
\int_{X\times Y}  \phi\,dP_{(X,Y)} = \int_X \Big( \int_Y \phi(x,y)\,
P(dy|X=x)\Big)\,dP_X(x).
\end{equation}
Of course the same results can be written in terms of a probability measure $\mu\in
P(\R^d\times \R^k)$ and its marginal in $\R^d$, $\mu_1$. In that case we let
$\mu_{|x}:=\nu(x,.)$ and   any $\mu$-integrable $\phi$  satisfies
\begin{equation}\label{FubiniCondDistrib2}
\int_{X\times Y}  \phi\,d\mu = \int_X \Big( \int_Y \phi(x,y)\,\mu_{|x}(y)
\Big)\,d\mu_1(x).
\end{equation}
The existence of $\nu$ is guaranteed  by Jirina's theorem. There are several classical
references in this subject, see for details \cite{Ash,CB,Kree,Stroock}.

\medskip

\subsection{The asymptotic behavior of solutions}

We are now ready to analyze how the interaction of stubborn agents   with those more
likely to change their opinions affects  the population's opinion dynamics. Indeed,
the agents with fixed opinion will drag the opinion of the rest of the individuals to
certain average of their own initial distribution, no matter the initial distribution
considered for the whole population, see Theorems \ref{ThmAymptoticConFirmes} and
\ref{TeoAsymptoticGeneral} below.

On the contrary, the asymptotic behavior when $q>0$ for every agent, which will be
studied in a future work, takes into account the values of the initial distribution of
every individual.

 Precisely, we consider an initial distribution $f_0\in P(K)$ of the form
\begin{equation}\label{InitialDistribConFirmes}
 f_0(w,p,q)dwdpdq = \alpha_0 f_0^0(w,p)dwdp\otimes \delta_{q=0}
                            + (1-\alpha_0) f_0^1,
\end{equation}
for some $\alpha_0\in (0,1]$, where $f_0^1\in P(K)$ is supported in $\{q\ge\eps\}$ for
some $\eps>0$, and $f_0^0$ is a probability measure on $[-1,1]\times [0,1]$. This
means that there exists a positive fraction $\alpha_0$ of stubborn people whose
opinion is distributed according to $f_0^0$, and that the parameters $(w,p,q)$ of the
rest of the population  verify $q\ge \eps$ and are determined by $f_0^1$.

Notice that the dynamics described in \eqref{Regla} deny  changes in  $(p,q)$ for
each agent and in consequence, the solution $f_t$ of \eqref{FPSinRuidoDebil} with
initial condition $f_0$ given in \eqref{InitialDistribConFirmes} will have the form
\begin{equation}\label{FormDistribConFirmes}
 f_t(w,p,q)dwdpdq = \alpha_0 f_0^0(w,p)dwdp\otimes \delta_{q=0} + (1-\alpha_0) f_t^1(w,p,q).
\end{equation}

We prove that in this case the non-stubborn agents share asymptotically the same
opinion $m_\infty$, which is completely determined by the opinion of the stubborn
individuals. Indeed, we shall see that $m_\infty$ is the mean opinion of the stubborn
people weighted by their persuasion power.

The occurrence of this fact is specially well observed if we assume that the marginal
in $(p,q)$ of the distribution of the opinion among the non-stubborn population,
$f_0^1(p,q)dpdq$, is given by a finite convex combination
$\sum_{i=1}^N \alpha_i\delta_{p=p_i,q=q_i}$  of Dirac masses. As we will see, the analysis of the
asymptotic behavior of the opinion distribution $f_t^i(w)dw$ of the population with
$(p,q)=(p_i,q_i)$
 can be conveniently  reduced to the study of a linear system of ordinary equations $M'=AM+B$ in $\R^N$.

It is known (see \cite{BGV}) that any probability measure $\mu\in P(\R^d)$ can be
approximated with high probability by the empirical measure $\hat \mu^N:=
\frac{1}{N}\sum_{i=1}^N \alpha_i\delta_{X_i}$, being $X_1,..,X_N$ $N$ random variables
identically distributed with law $\mu$. Then, it is  reasonable to think that the
results obtained for the discrete model enlighten the asymptotic behavior of the
general case, as it indeed occurs.

Therefore, we first examine the simplified discrete system to provide us with some
intuition, before  accomplishing the proof for any general initial distribution given
by (\ref{InitialDistribConFirmes}). This is the core of the following theorem.

\begin{thm} \label{ThmAymptoticConFirmes}
Let $f_0\in P(K)$ be an initial distribution defined as in
\eqref{InitialDistribConFirmes}, where
  the initial distribution $f_0^1$  of the variables $(w,p,q)$ corresponding to the non-stubborn population  has the form
\begin{equation} \label{Approximatef01}
 f_0^1 = \sum_{i=1}^N \alpha_i  g_0^i(w)dw \otimes \delta_{p=p_i, q=q_i},
\end{equation}
being $N\in \mathbb{N}$, $\alpha_1,..,\alpha_N> 0$ with $\alpha_1+..+\alpha_N=1$,
$q_1,..,q_N\in [\eps,1]$ all distinct, $p_1,..,p_N\in (0,1]$, and $g_0^1,..,g_0^N\in
P([-1,1])$. Its evolution in time,  $f_t^1$, verifies
\begin{equation}\label{ConvStubbornSimplified}
W_1(f_t^1,\delta_{m_0^0}\otimes \sum_{i=1}^N \alpha_i\delta_{p=p_i,q=q_i})\to 0,
\qquad \text{as $t\to +\infty$.}
\end{equation}
Here  $m_0^0$ denotes the mean opinion of the stubborn people weighted by the power of
persuasion, namely
\begin{equation}\label{DefMeanOpinionStubborn}
m_0^0 = \frac1{\pmean_{q=0}}\int_{-1}^1\int_0^1 pw\,df_0^0(w,p).
\end{equation}
\end{thm}

\begin{rem} An estimation of the velocity of convergence in \eqref{ConvStubbornSimplified} will be determined for the general case in Theorem~\ref{TeoAsymptoticGeneral}.

Furthermore, for an intuitive explanation for the fact that the opinion of the
non-stubborn agents converges to $m_0^0$, see Remark \ref{ExplConv} below. \end{rem}

\begin{proof}
Observe that the distribution  $f_t^1$ in \eqref{FormDistribConFirmes} has  the
form:
\begin{equation}\label{Deffti}
 f_t^1 = \sum_{i=1}^N \alpha_i g_t^i(w)dw\otimes \delta_{p=p_i,q=q_i},
\end{equation}
with $g_t^1,..,g_t^N \in P([-1,1])$. Notice that for any $i=1,..,N$,
$$ f_{t|p=p_i,q=q_i}^1 = g_t^i $$
and
\begin{equation}\label{Equfti}
 \p_t g_t^i(w) + \p_w((m_t-w)q_i \pmean g_t^i(w)) = 0,
\end{equation}
in weak formulation, namely for any $\phi\in C^1([-1,1])$,
\begin{equation}\label{WeakFormfti}
 \frac{d}{dt} \int_{-1}^{1} \phi(w)\,dg_t^i(w)
 = \int_{-1}^{1}   (m_t-w)q_i \pmean \phi'(w) \, dg_t^i(w).
\end{equation}
This follows from \eqref{FPSinRuidoDebil} extending $\phi$ to a smooth function with
support in $[-1,1]\times (p_i-\eta,p_i+\eta) \times (q_i-\eta,q_i+\eta)$ with $\eta>0$
small enough so that $(p_i-\eta,p_i+\eta) \times (q_i-\eta,q_i+\eta)$ does not contain
any other $(p_j,q_j)$.

\medskip

Let us study the behavior of $g_t^i$, $i=1,..,N$, as $t\to +\infty$. We claim that for
any $i=1,..,N$ and $t>0$,
\begin{equation}\label{FirmesMean}
 W_1(g_t^i,\delta_{m_t^i})  \le  2e^{-\eps \pmean t},
\end{equation}
where
$$ m_t^i:= \int_{-1}^1  w\,dg_t^i(w) $$
is the mean opinion of agents with $(p,q)=(p_i,q_i)$. Indeed, according to
Proposition~\ref{EquGenInv}, it follows from \eqref{Equfti} that the generalized
inverse $X_t^i$ of the cumulative distribution function corresponding to $g_t^i$
satisfies
$$ \p_t X_t^i(r) = (m_t-X_t^i(r)) q_i \pmean,  $$
for any $r\in (0,1]$. Then,
$$ \p_t (X_t^i(1)-X_t^i(r))^2  = -2q_i  \pmean  (X_t^i(1)-X_t^i(r))^2,  $$
 so that  by Gronwall's Lemma
\begin{equation}\label{Convergence200}
 X_t^i(1)-X_t^i(r) \le (X_0^i(1)-X_0^i(r))e^{-q_i \pmean t}
                           \le 2e^{-\eps \pmean t}.
\end{equation}
On the other hand, since
%$$\begin{array}{ll}
%\pmean_{|q=q_i}m_t^i
%&\D = \int_{-1}^{1}\int_0^1 pw \,df_{t|q=q_i}^1 (w,p)
% =\int_{-1}^1w\left(\int_0^1 p \,df_{t|q=q_i,w}^i(p)\right) df_t^i(w) \\
% [0.3cm]&\D =\int_{-1}^1w<p>_{|(q_i,w)}df_t^i(w)
%   =\int_0^1X_t^i(r)<p>_{|(q_i,X_t^i(r))}dr,
%\end{array}$$
%and also
%$$\begin{array}{ll}
%\D\int_0^1<p>_{|(
%q_i,X_t^i(r))}dr&\D=\int_{-1}^1<p>_{|(q_i,w)}df_t^i(w)\\&\D=\int_{-1}^1\int_0^1 pdf_{t|w}^1(p)df^i_t(w)\D=\int_{-1}^1\int_0^1 pdf_t^i(p,w)=<p>_{|q=q_i},
%\end{array}$$
$m_t^i = \int_0^1 X_t^i(r)\,dr$, it holds that $X_t^i(0^+)\leq m_t^i\leq X_t^i(1)$ for
all $t$. As a result,
\begin{eqnarray*}
 W_1(g_t^i,\delta_{m_t^i})
& \le&  \int_{-1}^1 |w-m_t^i|\,dg_t^i(w) = \int_0^1 |X_t^i(r)-m_t^i|\,dr \\
&\le & |X_t^i(1)-X_t^i(0^+)|,
\end{eqnarray*}
which, combined with \eqref{Convergence200}, gives \eqref{FirmesMean}.

\medskip

 In view of  \eqref{FirmesMean}, it is natural to study the asymptotic behavior of
 $m_t^i$, $i=1,..,N$. Taking $\phi(w)=w$ in \eqref{WeakFormfti}, we obtain
$$  \frac{d}{dt} m_t^i=(m_t-m_t^i)q_i \pmean.   $$
According to \eqref{FormDistribConFirmes}  and \eqref{Deffti}, we have
\begin{eqnarray*}
 \pmean m_t
& = & \int_{-1}^1\int_0^1\int_0^1 pw\,df_t(w,p,q)  \\
& = & \alpha_0 \pmean_{|q=0}m_0^0
   + (1-\alpha_0) \sum_{j=1}^N \alpha_j p_j m_t^j,
\end{eqnarray*}
where $m_0^0$ is defined in \eqref{DefMeanOpinionStubborn}. Thus for any $i=1,..,N$,
\begin{equation} \label{SistLinealSimplificado}
\begin{split}
 \frac{1}{q_i} \frac{d}{dt} m_t^i
&=  \Big((1-\alpha_0)\alpha_i p_i - \pmean \Big) m_t^i \\
&  \hspace{0.4cm}  + (1-\alpha_0)\sum_{k=1..N,\,k\neq i} \alpha_k p_k m_t^k
        + \alpha_0 \pmean_{|q=0} m_0^0.
\end{split}
\end{equation}
Introducing $M(t):=(m_t^1,..,m_t^N)^T$, this can be rewritten as
$$ M'(t)=AM(t) + B, $$
with $B = \alpha_0 \pmean_{|q=0} m_0^0 (q_1,..,q_N)^T$ and $A=(a_{ij})_{ij}$ with
$$ a_{ij} =
\begin{cases}
q_i \Big((1-\alpha_0) \alpha_i p_i - \pmean \Big)  & \qquad \text{if $j=i$} \\
 (1-\alpha_0) q_i \alpha_j p_j    &    \qquad \text{if $j\neq i$.}
\end{cases}
$$
The solution is explicitly
$$ M(t) = e^{tA}M(0) + \int_0^t e^{(t-s)A}B\,ds
           = e^{tA}(M(0)+A^{-1}B)- A^{-1}B.   $$
Notice that for any $i=1,..,N$,
$$ a_{ii} = -q_i\Big( \pmean_{|q=0}\alpha_0
            + (1-\alpha_0)\sum_{j=1,..,N,\,j\neq i} \alpha_j p_j \Big). $$
It is then easily seen that $A(1,..,1)^T=-\alpha_0 \pmean_{|q=0}(q_1,..,q_N)^T$, which yields that
$$ A^{-1}B = -m_0^0 (1, \ldots,1)^T. $$
According to Gerschgorin's disc theorem,
$$ \sigma(A) \subset  \bigcup_{i=1}^N D\Big(a_{ii},\sum_{k=1..N,k\neq i} |a_{ik}|\Big), $$
where $\sigma(A)$ denotes the spectrum of $A$ and $D(z,r)$ the disc in the complex
plane centered at $z$ of radius $r$. Thus, for any $i=1,..,N$,
\begin{equation}\label{signo} a_{ii}+\sum_{k=1..N,k\neq i} |a_{ik}| = -q_i \pmean_{q=0}\alpha_0
    \le -\eps \alpha_0 \pmean_{q=0}, \end{equation}
which implies that
\begin{equation} \label{BoundEigenvalues}
 \sigma(A) \subset \{z\in\C:\, Re(z)\le  -\eps \alpha_0 \pmean_{q=0}\}.
\end{equation}
We thus deduce that as $t\to +\infty$, $M(t)\to -A^{-1}B=m_0^0(1,..,1)^T$
exponentially fast. This together with \eqref{FirmesMean}   is enough to obtain that
$W_1(g_t^i,\delta_{m_0^0})\to 0$ for any $i=1,\ldots,N$. We are ready now to show
\eqref{ConvStubbornSimplified}. Using that $W_1$ combines properly with convex
combinations (see \cite{V}), we have
\begin{equation}\label{Convergence300}
\begin{split}
  W_1(f_t^1,\sum_{i=1}^N \alpha_i \delta_{m_0^0}\otimes \delta_{p=p_i,q=q_i})
& \le \sum_{i=1}^N \alpha_i W_1(g_t^i \otimes  \delta_{p=p_i,q=q_i},
                                             \delta_{m_0^0}\otimes \delta_{p=p_i,q=q_i}) \\
& \le \sum_{i=1}^N \alpha_i W_1(g_t^i,\delta_{m_0^0}),
\end{split}
\end{equation}
which goes to 0 as $t\to +\infty$. This completes the proof.
\end{proof}
\begin{rem} The problem without the presence of stubborn agents will be treated in a forthcoming work. Observe that in this case
 the inequality $(\ref{signo})$ is no longer
strictly negative. Therefore, this situation requires very different arguments.
\end{rem}

We now study the general case:

 \begin{thm} \label{TeoAsymptoticGeneral}
 Assume that the initial distribution has the form
 $$ f_0 = \alpha_0 f_0^0 + (1-\alpha_0) f_0^1, $$
 where $f_0^0\in P(K)$ is supported in $\{q=0\}$ and $f_0^1\in P(K)$ is
 supported in $\{q\ge \eps_0\}$.
Admit also that the map
$$ (p,q)\in [0,1]\times [0,\eps_0]\to f^1_{0|(p,q)}\in P([-1,1]) $$
 is globally Lipschitz for the $W_1$-distance:
there exists $L>0$ such that for any $(p,q)$, $(p',q')\in [0,1]\times [\eps_0,1]$,
 \begin{equation}\label{HypLipschitzDistribCond}
  W_1(f^1_{0|(p,q)},f^1_{0|(p',q')}) \le L(|q-q'|+|p-p'|).
 \end{equation}
  Then,
 \begin{equation}\label{AsymptoticFirmeNew}
 W_1(f_t^1,f_0^1(p,q)dpdq\otimes \delta_{m_0^0})
 \le 4e^{-\alpha_0\eps_0 \pmean _{|q=0} t},
 \end{equation}
where
\begin{equation}\label{Defm00}
m_0^0:= \int \frac{p}{\pmean_{|q=0}}w\,df_0^0(w,p)
\end{equation}
 is the mean opinion weighted by the normalized persuasion
power within the group of stubborn agents. Here $\pmean_{|q=0} = \int p\,df_0^0(p)$
stands for the mean value of $p$ among the stubborn agents.
 \end{thm}

\begin{rem}\label{ExplConv}
Let us give an intuitive motivation for the convergence of the opinion of the
non-stubborn agents to $m_0^0$. Suppose we know that there exists
$m_\infty:=\lim_{t\to +\infty} m_t$. In view of equation \eqref{FPSinRuido}, it seems
reasonable to conjecture that $f_t^1$ converges to $\delta_{m_\infty}(w)\otimes
f_0^1(p,q)dpdq$ so that $f_t\to \alpha_0f_0^0 +
(1-\alpha_0)\delta_{m_\infty}(w)\otimes f_0^1(p,q)dpdq$. In particular we can pass to
the limit in the definition of $m_t$ to obtain that
$$ \pmean m_\infty =  \pmean \lim_{t\to +\infty} m_t
= \alpha_0 \int pw \, df_0^0 + (1-\alpha_0)m_\infty \int p\,df^1_0(p,q).
$$
Moreover,
$$ \pmean = \alpha_0 \int p\,df_0^0(p) + (1-\alpha_0)\int p \, df_0^1(p)
                = \alpha_0  \pmean_{|q=0} + (1-\alpha_0)\int p \, df_0^1(p).
 $$
Furthermore,
$$ \alpha_0 \pmean_{|q=0} m_\infty = \alpha_0 \int pw \, df_0^0, $$
 which implies that $m_\infty = m_0^0$.
\end{rem}

\begin{rem}
Sufficient conditions on $f_0^1$ ensuring the regularity assumption
\eqref{HypLipschitzDistribCond} can be easily found. Suppose for instance that $f_0^1$
has a density in the sense that $f_0^1 = f_0^1(w,p,q)dwdpdq$ with  $f_0^1 \in L^1(K)$.
Then, $f_{0|(p,q)}^1(w) = \frac{f_0^1(w,p,q)}{f_0^1(p,q)}$ if  $f_0^1(p,q)\neq 0$. Let
us assume that
\begin{enumerate}
\item  $0<C\le f_0^1(p,q)\le C'<\infty$ for any $(p,q)$ such that $f_0^1(p,q)\neq 0$,
\item   there exists $C''>0$ such that
$$ |f_0^1(w,p,q)-f_0^1(w,p',q')|\le C''(|p-p'|+|q-q'|), $$
for any $w$ and any $(p,q), (p',q')$ with $f_0^1(p,q),f_0^1(p',q')\neq 0$.
\end{enumerate}
In that case $f_{0|(p,q)}^1$ verifies
$$ |f_{0|(p,q)}^1(w) - f_{0|(p',q')}^1(w)  |\le C''(|p-p'|+|q-q'|), $$
for any $w$ and any $(p,q), (p',q')$ with $f_0^1(p,q),f_0^1(p',q')\neq 0$.
Consequently,  for any $\phi:[-1,1]\to \R$ 1-Lipschitz and any $(p,q), (p',q')$ such that
$f_0^1(p,q),f_0^1(p',q')\neq 0$, there holds
\begin{equation*}
\begin{split}
  \int_{-1}^1 \phi(w)\,(df_{0|(p,q)}^1(w) - df_{0|(p',q')}^1(w))
& = \int_{-1}^1 \phi(w)\,(f_{0|(p,q)}^1(w) - f_{0|(p',q')}^1(w))dw  \\
& \le C''(|p-p'|+|q-q'|) \int_{-1}^1 |\phi(w)|\,dw.
\end{split}
\end{equation*}
Without loss of generality, it can be assumed that $\phi(-1)=0$, since the above
inequalities are still valid when adding a constant to $\phi$. Accordingly,
$\|\phi\|_\infty\le 2$. Taking the supremum over such $\phi$ in the above expression
 gives  \eqref{HypLipschitzDistribCond}  with $L=4C''$.
\end{rem}

In the course of the proof we will use the following envelope Theorem due to Milgrom
and Segal in \cite{SM}:

\begin{thm}\label{EnvelopeTheorem}
Let $X$ be a set. Consider the function $V(t):=\max_{x\in X} h(x,t)$, $t\in [0,1]$.
Admit that $h$ is absolutely continuous with respect to $t$ for any $x$  and there exists $b\in
L^1([0,1])$ such that $|\p_th(x,t)|\le b(t)$ for any $x\in X$ and almost any $t\in
[0,1]$. Then $V$ is absolutely continuous.

Assume in addition that $h$ is differentiable in $t$ for any $x\in X$ and that for any
$t\in [0,1]$ the set $X(t):=\text{argmax }h(.,t)$ is non-empty. In this case, for any
selection of $x^*(t)\in X(t)$ we have
$$ V(t) = V(0) + \int_0^t \p_th(x^*(s),s)\,ds. $$
\end{thm}

\subsection{The proof of Theorem \ref{TeoAsymptoticGeneral}}

We have all of the ingredients to show the asymptotic behaviour in the general case. For convenience, we  divide the proof in several
steps.

\begin{proof}[Proof of  Theorem \ref{TeoAsymptoticGeneral}]

 For any $t$ and any $p,q\in [0,1]\times
[\eps_0,1]$  denote by $f^1_{t|(p,q)}\in P([-1,1])$ the conditional distribution of
opinion among the agents with parameter $(p,q)$.

\begin{Step}
For any $(p,q)\in \text{supp }(f_0(p,q)dpdq)$, $f^1_{t|(p,q)}$ is the unique solution
to
\begin{equation}\label{EquOpinonCondq}
\left\{\begin{array}{rl}
 & \p_t f^1_{t|(p,q)} + \p_w((m_t-w)q \pmean f^1_{t|(p,q)}) = 0,\\
& f^1_{t=0|(p,q)} = f_{0|(p,q) },
\end{array}\right.
\end{equation}
in $C([0,+\infty),P([-1,1]))$.

Moreover, the function $(p,q)\to f^1_{t|(p,q)}$ is Lipschitz with respect to the
Wasserstein distance $W_1$. Namely, for any $(p,q),(p',q')\in [0,1]\times [\eps_0,1]$,
\begin{equation}\label{RegularidadCondDistrib2}
 W_1(f_{t|(p,q)}^1,f_{t|(p',q')}^1)\le  C_t(|q-q'|+|p-p'|).
\end{equation}

Furthermore, it fulfils
\begin{equation}\label{need} \int_K \phi \,df_t^1 = \int_0^1\int_0^1 \Big( \int_{-1}^1 \phi \,df_{t|(p,q)}^1(w)\Big) \,df_0^1(p,q), \qquad
\forall\,\phi\in C(K). \end{equation}

\end{Step}

\begin{proof}
 The existence of a unique solution to  (\ref{EquOpinonCondq}) is ensured by the results of Ca\~nizo,
 Carrillo y Rosado \cite{CCR}, see Remark \ref{RemCCR}.

Denote by $\phi_t$ the flow of the vector-field $(w,p,q)\to (q\pmean (m_t-w),0,0)$.
Since $m_t$ is considered to be  a known $C^1$ function, this flow can be rewritten as
$\phi_t(w,p,q)=(\phi^1_t(w,p,q),p,q)$ and $f_t^1=\phi_t\sharp f_0^1$  being the
push-forward measure defined as
$$
\int_K \psi(w,p,q)\,df_t^1(w,p,q)  = \int_K \psi(\phi_t(w,p,q))\,df_0^1(w,p,q),
$$
for all $\psi\in C(K)$.

This implies that for any $\psi\in C([-1,1])$ and $\phi\in C ([0,1]\times [0,1])$,
$$ \int_K \psi(w)\phi(p,q) \,df_t^1(w,p,q)
= \int_K \psi(\phi_t^1(w,p,q))\phi(p,q) \,df_0^1(w,p,q), $$ i.e.
\begin{eqnarray*}
  && \int_0^1\int_0^1 \phi(p,q)\Big( \int_{-1}^1 \psi(w) \, df_{t|(p,q)}^1(w)\Big) \,df_0^1(p,q)  \\
  && = \int_0^1\int_0^1 \phi(p,q) \Big( \int_{-1}^1 \psi(\phi_t^1(w,p,q))\,df_{0|(p,q)}^1(w)\Big)  \,df_0^1(p,q).
\end{eqnarray*}
 The arbitrariness of $\phi\in C([0,1]\times[0,1])$ yields, for any $t\ge 0$ and any  continuous function $\psi$, that
\begin{equation}\label{AsymptoticFirme30}
 \int_{-1}^1 \psi(w) \, df_{t|(p,q)}^1(w) =  \int_{-1}^1 \psi(\phi_t^1(w,p,q))\,df_{0|(p,q)}^1(w),
\end{equation}
for  almost any $(p,q)$, except for a $f_0^1(p,q)dpdq$-null set.

In particular, for any $k\in \mathbb{N}$, there exists a  $f_0^1(p,q)dpdq$-null set,
denoted as $A_{t,k}\subset [0,1]\times [0,1]$, for which  $\psi(w)=w^k$ verifies the
previous inequality  at any $(p,q)\in A_{t,k}^c$.

Note that $A_t:=\cup_{k\ge 0} A_{t,k}$ is a  $f_0^1(p,q)dpdq$-null set such that
\eqref{AsymptoticFirme30} holds for any polynomial $\psi$ and any $(p,q)\in A_t^c$.
The density of the polynomials in $C([-1,1])$  implies that indeed,
\eqref{AsymptoticFirme30} holds for any $\psi\in C([-1,1])$ and any $(p,q)\in A_t^c$
with $A_t$ of $f_0^1(p,q)dpdq$-null measure. This equality can then also be expressed as
\begin{equation}\label{RegularidadCondDistrib}
 f_{t|(p,q)}^1 = \phi_t^1(.,p,q)\sharp f_{0|(p,q)}^1,
\end{equation}
for any $(p,q)\in A_t^c$.

We would like this identity is fulfilled  for any  $(p,q)\in
\text{supp}(f_0^1(p,q)dpdq)$. So, let us fix certain $t\ge 0$ and  assume for the moment that there exists a constant $C_t>0$
depending only on $t$ such that for any $(p,q),(p',q')\in [0,1]\times [\eps_0,1]$,
\begin{equation}\label{RegularidadCondDistrib3}
 W_1( \phi_t^1(.,p,q)\sharp f_{0|(p,q)}^1, \phi_t^1(.,p',q')\sharp f_{0|(p',q')}^1)
 \le C_t(|q-q'|+|p-p'|).
\end{equation}
Note that this claim also shows the Lipschitz continuity stated in
(\ref{RegularidadCondDistrib2}).

Using the  decomposition $f_0^1(p,q)dpdq = f^{1,\;non-atom}_0+f^{1,\;atom}_0$, observe
that \eqref{RegularidadCondDistrib} is  also satisfied for any $(p,q)$ belonging to a
larger set, $ A_t^c\cup \{f_0^{1,\;atom}>0\}$, since $f^{1,\;atom}_0$ gives positive
mass to each of its atoms. This observation and the continuity given in
(\ref{RegularidadCondDistrib2}) conclude that  \eqref{RegularidadCondDistrib} is
satisfied in $A_t^c\cup \mbox{supp}(f^{1,\;atom})$.

It remains to modify the definition of $f_{t|(p,q)}^1$ at the variables $(p,q)\in A_t
\cap \text{supp}(f^{1,\;non-atom}_0)$ in such a way that  $f_{t|(p,q)}^1$ preserves
its continuity in $t$  and \eqref{RegularidadCondDistrib} holds for any $(p,q)\in
\text{supp}(f_0^1(p,q)dpdq)$.

Take first some $(p,q)\in A_t\cap \{f^{1,\;non-atom}_0>0\}$. Since $A_t$ has null
measure,  there exists a sequence $(p_k,q_k)\in A_t^c\cap \{f^{1,\;non-atom}_0>0\}$
such that $(p_k,q_k)\to (p,q)$. As a consequence of \eqref{RegularidadCondDistrib2},
$(f_{t|(p_k,q_k)})_k$ is a Cauchy sequence in the complete space $(P([-1,1],W_1)$,
hence it converges to some limit $g_{(p,q)}\in P([-1,1])$.
 Furthermore, \eqref{RegularidadCondDistrib2} ensures also that this limit does not depend on the approximating sequence $(p_k,q_k)_k$.

 We then declare $f_{t|(p,q)}^1:=g_{(p,q)}$ on $A_t\cap \{f^{1,\;non-atom}_0>0\}$.
That way $f^1_{t|(p,q)}$ is continuous and \eqref{RegularidadCondDistrib} holds for
any $(p,q)\in \{f^{1,\; non-atom}>0\}$. Proceed with $(p,q)\in \text{supp }(f^{1,\;
non-atom})$ similarly, taking an approximating sequence $(p_k,q_k)\in \{f^{1,\;
non-atom}>0\}$.

In conclusion, redefining $f_{t|(p,q)}^1$ on $f_0^1(p,q)dpdq$-null sets in such a way
that $f^1_{t|(p,q)}$ and the right hand side of \eqref{RegularidadCondDistrib} are
continuous  with respect to $(p,q)$, guarantees   that \eqref{RegularidadCondDistrib} and
\eqref{AsymptoticFirme30} are satisfied for any $t\ge 0$ and any $(p,q)\in \text{supp
}(f_0^1(p,q)dpdq)$. Moreover, it is clear that this modification is in accordance with
the application of Fubini's Theorem, thus  (\ref{need}) holds.

To conclude the proof, it remains to prove the claim \eqref{RegularidadCondDistrib3}.
Let $\psi:[-1,1]\to \R$ be 1-Lipschitz. Note that
\begin{equation} \label{AsymptoticFirme10}
\begin{split}
& \int \psi \,d(\phi_t^1(.,p,q)\sharp f_{0|(p,q)}^1 - \phi_t^1(.,p',q')\sharp f_{0|(p',q')}^1)  \\
 &= \int  \psi(\phi_t^1(w,p,q))\, df_{0|(p,q)}^1(w) - \int \psi(\phi_t^1(w,p',q'))\, df_{0|(p',q')}^1(w) \\
& = \int  \psi(\phi_t^1(w,p,q))-\psi(\phi_t^1(w,p',q'))\, df_{0|(p,q)}^1(w)  \\
 & \hspace{0.3cm}  + \int  \psi(\phi_t^1(w,p',q'))\,d(f_{0|(p',q')}^1 - f_{0|(p,q)}^1)(w)=I+II.
\end{split}
\end{equation}
The second term can be estimated using the definition of the Wasserstein distance
$W_1$,
\begin{eqnarray*}
&II& \leq Lip(\psi(\phi_t^1(.,p',q'))) W_1(f_{0|(p',q')}^1 , f_{0|(p,q)}^1)  \\
 &&   \le Lip(\phi_t^1(.,p',q')) L (|q-q'|+|p-p'|),
\end{eqnarray*}
where $L$ is the Lipschitz constant given by the assumption
\eqref{HypLipschitzDistribCond}. On the other hand,  the first term in
\eqref{AsymptoticFirme10} can be bounded  as
$$ I\leq \max_{|w|\le 1}|\psi(\phi_t^1(.,p,q))-\psi(\phi_t^1(.,p',q'))|
 \le  \max_{|w|\le 1}|\phi_t^1(.,p,q)-\phi_t^1(.,p',q')|. $$
Summarizing,
\begin{eqnarray*}
&&  \int \psi \,d(\phi_t^1(.,p,q)\sharp f_{0|(p,q)}^1 - \phi_t^1(.,p',q')\sharp f_{0|(p',q')}^1)  \\
&& \qquad \le  \max_{|w|\le 1}|\phi_t^1(w,p,q)-\phi_t^1(w,p',q')|
    + Lip(\phi_t^1(.,p',q'))  L (|q-q'|+|p-p'|).
\end{eqnarray*}
At this stage, recall that
$$ \phi_t^1(w,p,q) = w + q\pmean \int_0^t (m_s-\phi_s^1(w,p,q))\,ds,  $$
therefore,
\begin{eqnarray*}
&& \phi^1_t(w,p,q)-\phi^1_t(w,p',q')   = (q-q')<p>\int_0^t( m_s - \phi^1_s(w,p',q'))\,ds \\
&& \qquad + q\pmean \int_0^t (\phi_s(w,p',q')- \phi_s(w,p,q))\,ds.
\end{eqnarray*}
Taking now into account that $|\phi_t^1(w,p,q)|\le 1$ for any $(w,p,q)$, we infer that

\begin{eqnarray*}
 |\phi_t(w,p,q)-\phi_t(w,p',q')|
\le  2t |q-q'| + \int_0^t |\phi_s(w,p',q')- \phi_s(w,p,q)|\,ds.
\end{eqnarray*}
With the use of Gronwall's lemma this yields
$$ |\phi_t^1(w,p,q)-\phi_t^1(w,p',q')|\le 2te^t|q-\hat  q|. $$
Similar arguments prove that for $w,\tilde w\in [-1,1$ and $p,q\in [0,1]$,
$$ |\phi_t^1(w,p,q) - \phi_t^1(\tilde w,p,q)| \le e^t|w-\tilde w|, $$
hence $ Lip(\phi_t^1(.,p,q))\le e^t$. In conclusion, we have shown that
$$  \int \psi \,d(\phi_t^1(.,p,q)\sharp f_{0|(p,q)}^1 - \phi_t^1(.,p',q')\sharp f_{0|(p',q')}^1)
\le C(t)(|q-q'|+|p-p'|).  $$ The desired claim \eqref{RegularidadCondDistrib3} follows
now taking the supremum among all functions $\psi$ 1-Lipschitz.
\end{proof}

\begin{rem}

 It would be natural to conjecture  that the density $f_{t|(p,q)}^1$, modified in $f_0^1(p,q)dpdq$-null set as in Theorem~\ref{EnvelopeTheorem}, still defines a conditional density. Indeed, it is straightforward to see that
 $$\mu_w(p,q):=f_{t|(p,q)}^1(w)\in P([0,1]\times[0,1]) \mbox{ for any } w\in [-1,1],$$
 and $P(X\in A;Y\in B) =  \int_A\int_B d\mu_w(p,q)\,df_{t|(p,q)}^1(w)$ for any $A\subset [-1,1]$, $B\subset [0,1]\times[0,1]$ Borel sets.
 However, the fact that $f_t^1(.,B)$ is measurable for any $B\subset [0,1]\times[0,1]$ Borel, is not so immediate, and nevertheless is out of the scope of our results. In particular, for our proof it suffices with $(\ref{need})$.

\end{rem}

We denote by
\begin{equation}\label{eme}m(t,p,q)=\int_{-1}^1 w \, df_{t|(p,q)}^1(w),
\end{equation}
the mean opinion among the agents with parameter $p,q\in [0,1]\times [\eps_0,1]$.

\begin{Step} There holds
\begin{equation}\label{AsymptoticFirmeConvMedia}
 W_1(f_{t|(p,q)}^1,\delta_{m(t,p,q)})\le 2e^{-\eps_0 \pmean t}.
\end{equation}
\end{Step}

\begin{proof}
It can be deduced analogously to \eqref{FirmesMean}.
\end{proof}

In view of the previous step, it is natural to study the asymptotic behavior of the
function $m(t,.)$ as $t\to +\infty$.
\begin{Step} For any $t\ge 0$ and any $(p,q)\in [0,1]\times [\eps_0,1]$ the function
$m(t,p,q)$ defined in (\ref{eme})
 satisfies
\begin{equation} \label{SistemaContinuo2}
\begin{split}
 \p_tm(t,p,q)
& =  q\alpha_0 \pmean_{|q=0} \Big[ m_0^0 -  m(t,p,q)\Big]  \\
&   \hspace {0.4cm} + (1-\alpha_0)q \int_K p'\Big[ m(t,p',q')-m(t,p,q)\Big]\,
df_0^1(w,p',q').
\end{split}
\end{equation}

\end{Step}
\begin{proof}
Given that $f_{t|(p,q)}^1$ fulfills  \eqref{EquOpinonCondq} for any $(p,q)\in
[0,1]\times [\eps_0,1]$, in particular
\begin{eqnarray}\label{nombre}
 \p_tm(t,p,q)
& =& \frac{d}{dt} \int_{-1}^1 w \, df_{t|(p,q)}^1(w) = q\pmean \int_{-1}^1 (m_t-w)\,
df_{t|(p,q)}^1(w) \\\nonumber
 & = & q\pmean (m_t-m(t,p,q)).
\end{eqnarray}
Moreover according to the definition of $m_t$,
\begin{eqnarray*}
 \pmean m_t
& = & \int_K pw\,df_t(w,p,q)  \\
& = & \alpha_0 \int_{-1}^1\int_0^1 pw \,df_0^0(w,p) + (1-\alpha_0) \int_K pw\,
df_t^1(w,p,q),
\end{eqnarray*}
being
\begin{eqnarray*}
 \int_K pw\, df_t^1(w,p,q)
& = & \int_{0}^1 \int_{0}^1 p \Big( \int_{-1}^1  w \, df_{t|(p,q)}^1(w) \Big) \, df_0^1(p,q)  \\
&  = & \int_{0}^1 \int_{0}^1 p m(t,p,q) \,df_0^1(p,q).
\end{eqnarray*}
Denote as $\pmean_{|q=0}:=\int p\,df_0^0(p,w) $, that is, the mean value of $p$ among
the agents with $q=0$. We have
\begin{eqnarray*}
\pmean & = & \int_K p\,df_t(w,p,q)
%  = \alpha_0 \int p\,df_0^0(p,w) + (1-\alpha_0) \int p \, df_t^1(w,p,q) \\
 = \alpha_0 \pmean_{|q=0} +     (1-\alpha_0) \int_K p \, df_t^1(w,p,q).
\end{eqnarray*}
In terms of $\pmean$ and $\pmean m_t$ equation (\ref{nombre}) is equivalent to
\begin{eqnarray*}
 \frac{1}{q} \p_tm(t,p,q)
& = & \alpha_0 \Big[ \int_{-1}^1\int_0^1 pw \,df_0^0(w,p) - \pmean_{|q=0} m(t,p,q)\Big]  \\
& &     + (1-\alpha_0) \int_K p'\Big[ m(t,p',q')-m(t,p,q)\Big]\, df_0^1(w,p',q'),
\end{eqnarray*} which in view of  the definition  of $m_0^0$ in \eqref{Defm00},
can be rewritten as (\ref{SistemaContinuo2}).
\end{proof}

 At this stage, our aim is to determine the behavior as $t\to +\infty$ of the solution $m(t,p,q)$ to the linear system (\ref{SistemaContinuo2}), which is exactly the system appearing in \eqref{SistLinealSimplificado},  when
$f^1_0$ had the special form $f_0^1 = \sum_{i=1}^N \alpha_i g_0^i(w)dw\otimes
\delta_{p=p_i,q=q_i}$.

\medskip

The following step proves that $m(t,p,q)$ is Lipschitz continuous in $(p,q)$ uniformly
in $t$.

\begin{Step}
For any $\eps_0<\eps<2/(L\pmean)$ (where $L$ is given in
\eqref{HypLipschitzDistribCond}), and for any $(p,q),(p',q')\in [0,1]\times
[\eps_0,1]$,  there holds
\begin{equation} \label{m_Lipschitz}
|m(t,p,q)-m(t, p',q')|\le \frac{2}{\eps\pmean}(|q-q'|+|p-p'|).
\end{equation}
\end{Step}

\begin{proof}
Using \eqref{SistemaContinuo2}   we have
\begin{eqnarray*}
&& \p_t[m(t,p,q)-m(t, p',q')]   \\
&& = \alpha_0 \pmean_{|q=0}  m_0^0(q-q') - \alpha_0 \pmean_{|q=0}  (q-q')m(t,p,q) \\
&& \hspace{0.3cm} -\alpha_0\pmean_{|q=0}  q' [m(t,p,q)-m(t,p',q')]  \\
&& \hspace{0.3cm} + (1-\alpha_0) (q- q')
  \int_0^1\int_0^1 \tilde p [m(t,\tilde p, \tilde q)-m(t,p,q)]\,df_0^1(\tilde p,\tilde q)   \\
&& \hspace{0.3cm}  - (1-\alpha_0) q' [m(t,p,q)-m(t,p',q')] \int_0^1\int_0^1 p\,df_0^1( p, q)  \\
%  && = (q-q')
% \Big\{ \alpha_0<p>_{|q=0}[m_0^0-m(t,p,q)]  \\
% &&  \hspace{1.9cm}   +(1-\alpha_0) \int \tilde p [m(t,\tilde p,\tilde q)-m(t,p,q)]\,df_0^1(\tilde q)  \Big\} \\
% && \hspace{0.3cm}  - [m(t,p,q)-m(t,p',q')] \Big\{ \alpha_0 <p>_{|q=0} q' + (1-%\alpha_0)q' \int p\,df_0^1 \Big\}  \\
&& = (q-q')
 \Big\{ \alpha_0 \pmean_{|q=0}[m_0^0-m(t,p,q)]  \\
 &&  \hspace{1.9cm}
  + (1-\alpha_0) \int_0^1\int_0^1 \tilde p [m(t,\tilde p,\tilde q)-m(t,p,q)]\,df_0^1(\tilde p,\tilde q)  \Big\} \\
 && \hspace{0.3cm}  - [m(t,p,q)-m(t, p',q')] q' \pmean.
\end{eqnarray*}
Consequently,
\begin{eqnarray*}
&& \frac12  \frac{\p}{\p t} |m(t,p,q)-m(t,p',q')|^2   \\
&& = (q-q')\Big[m(t,p,q)-m(t,p',q') \Big]
  \Big\{ \alpha_0 \pmean_{|q=0}[m_0^0-m(t,p,q)]  \\
 &&  \hspace{3cm}
 +(1-\alpha_0) \int \tilde p [m(t,\tilde p,\tilde q)-m(t,p,q)]\,df_0^1(\tilde p, \tilde q)  \Big\} \\
 && \hspace{0.3cm}  - |m(t,p,q)-m(t,p',q')|^2 q' \pmean.
\end{eqnarray*}
Recalling that $|m_0^0|,|m(t,p,q)|\le 1$, it is straightforward to see that
$$
\Big| \alpha_0 \pmean_{|q=0}[m_0^0-m(t,p,q)]
  + (1-\alpha_0) \int \tilde p [m(t,\tilde p,\tilde q)-m(t,p,q)]\,df_0^1(\tilde p,\tilde q)  \Big|\leq 2.
$$
The fact that $q'\ge\eps_0$ allows to deduce that
\begin{eqnarray*}
&& \frac12  \frac{\p}{\p t} |m(t,p,q)-m(t,p',q')|^2   \\
&& \le 2|q-q'|\Big|m(t,p,q)-m(t,p',q') \Big| - \eps_0 \pmean |m(t,p,q)-m(t,p',q')|^2 \\
&& \le 2(|q-q'|+|p-p'|)\Big|m(t,p,q)-m(t,p',q') \Big| \\
&& \hspace{0.3cm}    - \eps_0 \pmean |m(t,p,q)-m(t,p',q')|^2.
\end{eqnarray*}
Let $u(t)\ge 0$ be the solution to
$$\left\{\begin{array}{ll}
   u'(t) &= 4(|q-q'|+|p-p'|) \sqrt{u(t)} - 2\eps_0 \pmean u(t), \\[0.02cm]
    u(0)&= |m(0,p,q)-m(0,p',q')|^2.
\end{array}\right.$$
Note that $|m(t,p,q)-m(t,p',q')|^2 \le u(t)$. Moreover, writing the equation for $u$
as
$$ u'(t) = 2 \eps_0 \pmean \sqrt{u(t)} \Big( u^* -  \sqrt{u(t)} \Big),$$
where
$$
u^*:= \frac{2(|q-q'|+|p-p'|)|}{\eps_0 \pmean}, $$ we see that  if $u(0)\le (u^*)^2$,
then $u(t)\le (u^*)^2$ for any $t$ and $u(t)\to (u^*)^2$. Observe that in view of the
assumption \eqref{HypLipschitzDistribCond},
\begin{eqnarray*}
 u(0) & = & |m(0,p,q)-m(0,p',q')|^2
 = \Big| \int_{-1}^1 w\,df^1_{0|(p,q)}(w) - \int_{-1}^1 w\,df^1_{0|(p',q')}(w) \Big|^2 \\
& \le & \Big( W_1(f^1_{0|(p,q)} , f^1_{0|(p',q')}  )\Big)^2
 \le  L^2(|q-q'|+|p-p'|)^2.
\end{eqnarray*}
Taking $\eps>0$ such that $\eps_0<\eps<2/(L\pmean)$, ensures that $u(0)\le (u^*)^2$
and thus  $u(t)\le (u^*)^2$ for any $t$. It follows then $ |m(t,p,q)-m(t,p',q')|^2 \le
u(t)\le (u^*)^2 $, which proves (\ref{m_Lipschitz}).

\end{proof}

We have all of the ingredients to show the convergence of  $m(t,p,q)$ to $m_0^0$:

\begin{Step}
For any $(p,q)\in \text{supp}\,(f_0^1(p,q)dpdq)$ and any $t\ge 0$ it holds that
$$  |m(t,p,q)-m_0^0|
\le \Big(\max_{(p,q)\in \text{supp}(f_0^1)} |m(0,p,q)-m_0^0| \Big)
 e^{-\eps_0\alpha_0 \pmean_{|t=0} t}.  $$
\end{Step}

\begin{proof}
Relation \eqref{SistemaContinuo2} implies  that for any $q\in [\eps_0,1]$ and $t\geq
0$
\begin{equation} \label{AsymptoticFirme20}
\begin{split}
& \frac12  \frac{\p}{\p t} |m(t,p,q)-m_0^0|^2  \\
&  = \p_t m(t,p,q) [m(t,p,q)-m_0^0]   \\
&  = -q\alpha_0 \pmean_{|q=0} [m_0^0-m(t,p,q)]^2  \\
 & \hspace{0.4cm} + q(1-\alpha_0)[m(t,p,q)-m_0^0]
   \int_0^1\int_0^1 \tilde p\Big[m(t,\tilde p,\tilde q)-m(t,p,q)\Big]\,df_0^1(\tilde p,\tilde q).
\end{split}
\end{equation}
  In particular, choosing $(p,q)=(p^*,q^*)$ a maximum point for $|m(t,.)-m_0^0|$ (its
existence is ensured since $\text{supp}(f_0^1(p,q)dpdq)$ is compact and $m(t,.)$ is continuous). Then,
\begin{eqnarray*}
&& \frac12  \frac{\p}{\p t} |m(t,.)-m_0^0|^2_{|(p^*,q^*)}  \\
& & =  -q^*\alpha_0 \pmean_{|q=0}  [m_0^0-m(t,p^*,q^*)]^2  \\
&& \hspace{0.4cm} + q^*(1-\alpha_0)[m(t,p^*,q^*)-m_0^0]
        \int_0^1\int_0^1 \tilde p\Big[ (m(t,\tilde p,\tilde q)-m_0^0) + (m_0^0-m(t,p^*,q^*))\Big]\,df_0^1(\tilde p,\tilde q)  \\
&  & = -q^*\alpha_0 \pmean_{|q=0} [m_0^0-m(t,p^*,q^*)]^2  \\
&& \hspace{0.4cm} + q^*(1-\alpha_0)[m(t,p^*,q^*)-m_0^0]
        \int_0^1\int_0^1 \tilde p\Big[ m(t,\tilde p,\tilde q)-m_0^0 \Big]\,df_0^1(\tilde p,\tilde q)   \\
&& \hspace{0.4cm} - q^*(1-\alpha_0)[m(t,p^*,q^*)-m_0^0]^2
 \int_0^1\int_0^1 \tilde p\,df_0^1(\tilde p,\tilde q)\\
 && =I+II+III.
\end{eqnarray*}
The choice of $q^*$ assures that
$$ II\leq q^*(1-\alpha_0)|m(t,p^*,q^*)-m_0^0|^2  \int_0^1\int_0^1 \tilde p \,df_0^1(\tilde p,\tilde q)=-III.  $$
The cancellation of these two terms gives
\begin{eqnarray} \label{DerMax100}
 \frac{\p}{\p t} |m(t,.)-m_0^0|^2_{|(p^*,q^*)}
\le -2\eps_0\alpha_0 \pmean_{|q=0} |m_0^0-m(t,p^*,q^*)|^2.
\end{eqnarray}

Denote $V(t)=\max_{(p,q)\in \text{supp}(f_0^1)} h(t;(p,q)) $ with
$h(t;(p,q))=|m(t,p,q)-m_0^0|^2$, which in $t$ is a $C^1$ function since $m$ is $C^1$
in $t$. Moreover, by \eqref{AsymptoticFirme20} it holds that $|\p_th(t;(p,q))|\le C$.
Now the envelope Theorem~\ref{EnvelopeTheorem} applies  to obtain that
 $V$ is absolutely continuous with derivative
$$ V'(t) = \p_t \Big(|m(t,q^*)-m_0^0|^2\Big)  \qquad a.e.. $$
Thus, in view of \eqref{DerMax100},
$$ V'(t)\le -2 \eps_0\alpha_0 \pmean_{|q=0} V(t) $$
and as a result
$$ V(t) \le V(0)e^{-2\eps_0\alpha_0 \pmean_{|q=0}  t},$$
which completes the proof.
\end{proof}

We are now in position to accomplish the proof of Theorem \ref{TeoAsymptoticGeneral}.
The previous Step ensures that for any $t\ge 0$ and any $(p,q)\in \text{supp
}f_0^1(p,q)dpdq$,
$$ W_1\Big(\delta_{m(t,p,q)},\delta_{m_0^0}\Big)
   = |m(t,p,q)-m_0^0| \le 2e^{-\eps_0\alpha_0 \pmean_{|q=0} t}. $$
According to \eqref{AsymptoticFirmeConvMedia} and noticing that $\pmean \ge
\alpha_0\pmean_{|q=0}$, we infer that
$$ W_1\Big(f^1_{t|(p,q)},\delta_{m_0^0}\Big)
     \le 2e^{-\eps_0\alpha_0 \pmean_{|q=0} t} + 2e^{-\eps_0 \pmean t}
        \le  4e^{-\eps_0\alpha_0 \pmean_{|q=0} t}.  $$

We  now claim that
$$ W_1\Big(f^1_t,\delta_{m_0^0}\otimes f_0^1(p,q)dpdq\Big)
 \le  4e^{-\eps_0\alpha_0 \pmean_{|q=0} t}. $$
Indeed let $\psi:K\to \R$ be 1-Lipschitz. Then,
\begin{eqnarray*}
 &&\int_K \psi\,\Big(df_t^1-\delta_{m_0^0}\otimes f_0^1(p,q)dpdq\Big)  \\
 && = \int_0^1 \int_0^1
 \Big( \int_{-1}^1 \psi(w,p,q)\, (df_{t|(p,q)} - \delta_{m_0^0}) \Big) \, df_0(p,q).
\end{eqnarray*}
Since $\psi(.,p,q)$ es 1-Lipschitz, the inner integral is bounded above by $W_1(
f_{t|(p,q)}, \delta_{m_0^0})$, which implies  that
\begin{eqnarray*}
 \int_K \psi\,\Big(df_t^1-\delta_{m_0^0}\otimes f_0^1(p,q)dpdq\Big)
& \le & 4e^{-\eps_0\alpha_0 \pmean_{|q=0} t} \int_0^1 \int_0^1  df_0(p,q)  \\
& = & 4e^{-\eps_0\alpha_0 \pmean_{|q=0} t}.
\end{eqnarray*}
The claim follows taking supremum over all functions $\psi$ 1-Lipschitz. The proof of
the theorem is now complete.
\end{proof}

\section{Computational experiments}

We close the paper with some agent based simulations.

The numerical experiment  considers a population of $N=10000$ agents with
$\alpha_0=60\%$  stubborn agents. We take such a high proportion to speed up the
computations in view of \eqref{AsymptoticFirmeNew} whereas it does not change the
value of $m_\infty$.

Initially,
\begin{itemize}
\item each non-stubborn agent has opinion chosen uniformly at random in $[0.3;1]$,
parameter $q$ is chosen uniformly at random in $[0.2;1]$ and we set $p=1-q$,

 \item one third of
the stubborn agents has $p=0.6$ and opinion chosen at random in $[-0.8;-0.6]$
uniformly,  whereas the others  have $p=0.2$ and opinion chosen uniformly at random in
$[0.4;0.8]$.
\end{itemize}
Notice in particular  that
\begin{eqnarray*}
 \pmean_{q=0} & = & \frac13\times 0.6 + \frac23\times 0.2 = \frac13,  \\
 \pmean & = & \alpha_0 \pmean_{q=0} + (1-\alpha_0)\int p\,df_0^1
           = 0.6*\frac13 + 0.4*(1-0.6) \\
           & = & 0.36, \\
\int pw \,df^0_0(w,p) & = &
  \frac13 \times 0.6 \times (-0.7) + \frac23 \times 0.2 \times 0.6.
\end{eqnarray*}
It follows that
$$ m_\infty = \int \frac{pw}{\pmean_{q=0}} \,df^0_0(w,p) \approx -0.18. $$

We then  let the agents interact following the rules with $\gamma = 0.01$. We depict in
the figures in Table \ref{DensityEvol} the density of $(w,q)$ parameters among the
non-stubborn population for different times in gray scale (the whiter is the graphic,
the higher is the density). This picture clearly reveals the convergence of the density of
opinion $f_{t|q}(w)dw$ among the population with $q$ towards its mean value $m_q(t)$
(at a faster pace for higher $q$ as predicted), and then the displacement of the
curve-like density to a vertical segment located at $w \approx -0.18$. This is in
complete agreement with the theoretical value $m_\infty$ given above.

\begin{table}[!ht]\label{DensityEvol}
\begin{center}
\begin{tabular}{cc}
$$ \includegraphics[width=0.5\linewidth, height=0.5\linewidth]{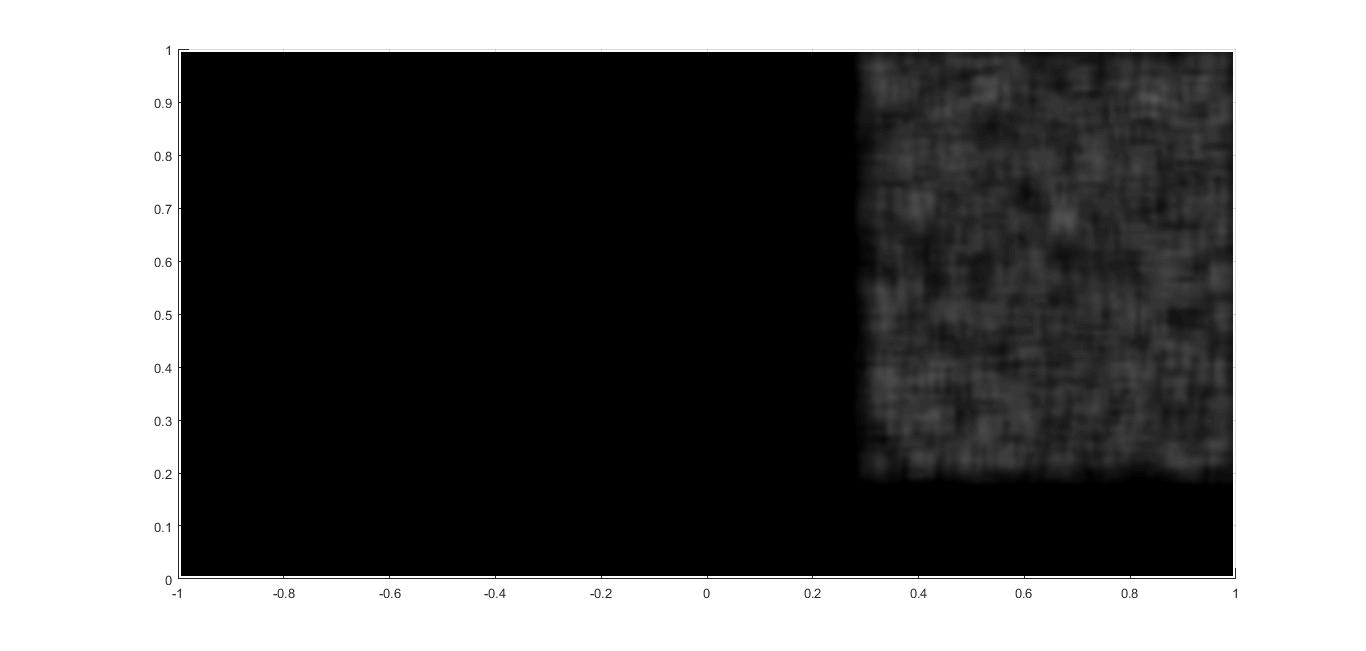} $$
&
$$ \includegraphics[width=0.5\linewidth, height=0.5\linewidth]{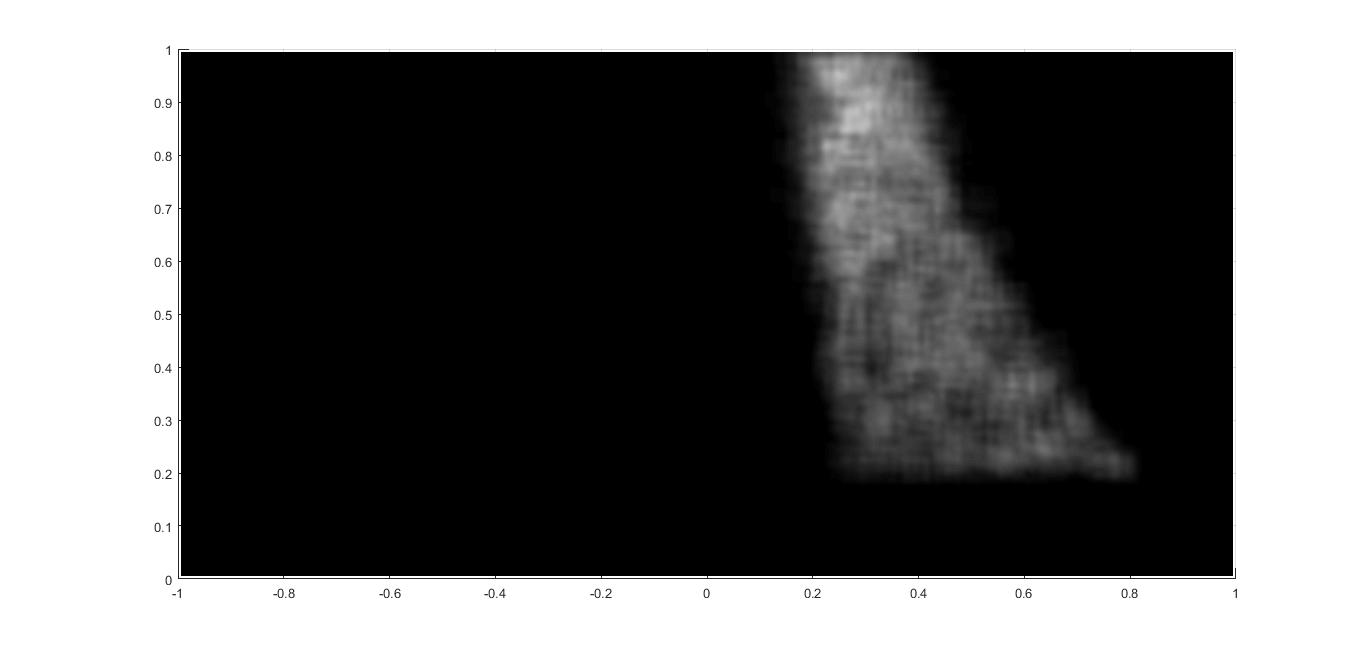} $$ \\
$$ \includegraphics[width=0.5\linewidth, height=0.5\linewidth]{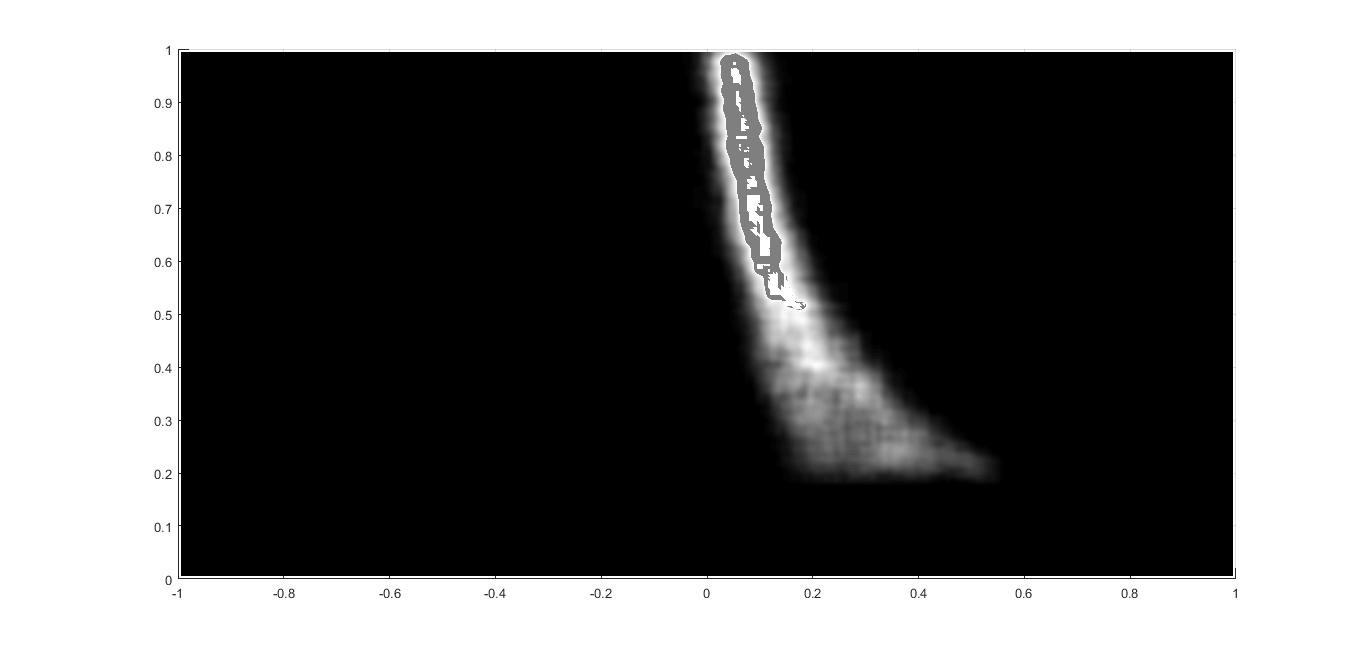} $$
&
$$ \includegraphics[width=0.5\linewidth, height=0.5\linewidth]{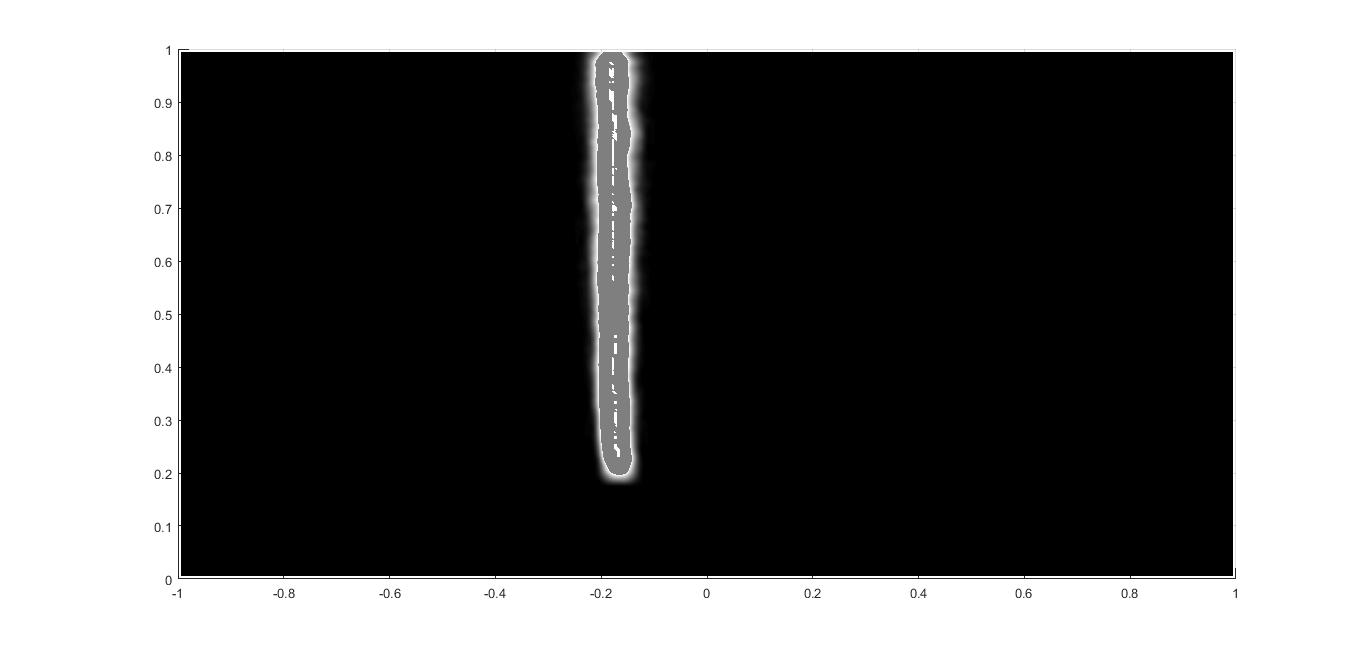} $$ \\
\end{tabular}
\end{center}
\caption{From left to right and top to bottom, the figures represent the density of
$(w,q)$ parameters among the non-stubborn population in gray scale (the whiter is the
graphic, the higher is the density) with opinion $w$ in the horizontal axe, and $q$ in
the vertical axe after respectively 0,37,96,700 $\times 50 000$ interactions.
  The initial values of $(w,p,q)$ are those given in the text. }
\end{table}

\appendix

\section{Existence of a unique solution to the Boltzmann equation. Proof of Theorem~\ref{ThmExistenceBoltzmann}}
The existence of a unique solution to the Boltzmann equation results as an application of the classical Banach Fixed
Point Theorem, as it is sketched  in \cite{CI}. We provide here a detailed proof, for the reader's convenience.
\begin{proof}
 Let us  first introduce some notations. If $f,g\in M(K)$ are given measures, we define
a finite measure $Q(f,g)$ by
\begin{equation}\label{defiQ}\begin{array}{lll}
 (Q(f,g),\phi)
& = & \D\frac12 \int_{K^2\times B}
(\phi(\varpi')-\phi(\varpi))\,df(\varpi)dg(\varpi_*)d\theta(\eta)
\\[0.3cm]
&& + \D\frac12 \int_{K^2\times B}
(\phi(\varpi')-\phi(\varpi))\,dg(\varpi)df(\varpi_*)d\theta(\eta),
\end{array}
\end{equation}
for any $\phi\in C(K)$. Notice that
$$
 |(Q(f,g),\phi)|\le 2\|\phi\|_\infty \|f\|\|g\|,
$$
where $\|f\|$ and $\|g\|$ denote the total variation norm (\ref{norm}) of $f$ and $g$,
respectively. Consequently, the total variation norm of the measure $Q(f,g)$ verifies
that
\begin{equation}\label{BoundBoltz2}
 \|Q(f,g)\|\le 2\|f\|\|g\|.
\end{equation}
Observe for future use that $ Q(f,f)-Q(g,g)=Q(f+g,f-g)$ which yields
\begin{equation}\label{BoundBoltz3}
\|Q(f,f)-Q(g,g)\|\le 2\|f+g\|\|f-g\|.
\end{equation}

 Fix some $T>0$ to be chosen later on.
 Denote by $C_T:=C([0,T],P(K))$ the space of functions from $[0,T]$ with values in  $P(K)$
  being  continuous for the total variation norm.
  The estimate (\ref{BoundBoltz2}) ensures that for all
 $f,g\in C_T$, it holds  $\|Q(f_s,g_s)\|\le C$ for any $s\in [0,T]$.
Thus $\int_0^t \|Q(f_s,g_s)\|\,ds$ is finite and we can consider the integral
$\int_0^t Q(f_s,g_s)\,ds$ in the Bochner sense. Since $P(K)$ is separable (because $K$
is compact), the integration is indeed understood in the Pettis sense. In particular,
\begin{equation}\label{PruebaBoltzDefIntegral}
 (\int_0^t Q(f_s,g_s)\,ds,\phi)=\int_0^t (Q(f_s,g_s),\phi)\,ds, \qquad\text{for any $\phi\in C(K)$. }
\end{equation}

\medskip

Under  these notations,  \eqref{BoltzDebil} can be rewritten  as
$$  \int_K \phi\,df_t = \int_K \phi\,df_0 + \int_0^t (Q(f_s,f_s),\phi)\,ds, $$
i.e.
\begin{equation}\label{Boltz}
f_t = f_0 + \int_0^t Q(f_s,f_s) \,ds =: J(f)(t).
\end{equation}
Our purpose is to find a fixed point of $J$ in the closed subspace $X_T$ of $C_T$ defined as
$$ X_T:=\{ f\in C_T: f(0)=f_0 \mbox{ and }\, \max_{0\le t\le T} \|f_t\|\le 2\|f_0\|\}, $$
where $T$ is sufficiently small. We endow  $X_T$ with the sup-norm given by
$\|f\|_{X_T}=\max_{0\le t\le T} \|f_t\|$.  For any $f\in X_T$, a direct application of
the Dominated Convergence Theorem ensures that $J(f)\in C_T$. Moreover, in view of
\eqref{BoundBoltz2}, we get
\begin{eqnarray*}
 \|J(f)(t)\| & \le & \|f_0\| + \int_0^t \|Q(f_s,f_s)\|\,ds \le \|f_0\| + 2T\max_{0\le t\le T} \|f_s\|^2  \\
               & \le & \|f_0\|+8T\|f_0\|^2.
\end{eqnarray*}
Taking $T\le 1/(8\|f_0\|)$ guarantees  that $J(f)\in X_T$. Next we prove that $J$ is
in fact a strict contraction. Recall that by \eqref{BoundBoltz3}  we know
\begin{eqnarray*}
 \|J(f)(t)-J(g)(t)\|
& \le & \int_0^t \|Q(f_s,f_s)-Q(g_s,g_s)\|\,ds  \\
& \le & 2\int_0^t \|f_s+g_s\|\|f_s-g_s\|\,ds  \\
& \le & 8T\|f_0\|\|f-g\|.
\end{eqnarray*}
The choice e.g. $T= 1/(16\|f_0\|)$ provides that $\|J(f)-J(g)\|\le \frac12 \|f-g\|$.
The existence of a unique fixed point of $J$ in $X_T$ consequently follows.

\medskip

Taking $\phi=1$ in \eqref{BoltzDebil} and recalling that $f_0\in P(K)$ shows that
$\int_K f_t(\varpi)=1$. It just remains to see that  $f_t\ge 0$ to infer that $f_t\in
P(K)$ with $\|f_t\|=\|f_0 \|=1$. At this point, we
 could then repeat the previous argument
to extend $f_t$ to $[T,2T]$, $[2T,3T]$, and so on,  and conclude the existence proof.
 Proposition~\ref{ProofExistenciaBoltzLemmaPositivity} below is devoted  to prove the
non-negativity of $f$, which completes this proof.
\end{proof}

\begin{rem}
\noindent Bearing in mind that $f$ is continuous in time, it is no difficult to see
that $ \int_K \phi(w)df_t(w) $ is a $C^1$ function with respect to $t$, whose  derivative is specified by
\eqref{BoltzDebil2}.

Indeed, with the notations introduced in the previous proof, it holds that
\begin{eqnarray*}
\Big  \|\frac{f_{t+h}-f_t}{h}-Q(f_t,f_t)\Big \|
& =& \Big\|\frac{1}{h}\int_t^{t+h} Q(f_s,f_s)\,ds - Q(f_t,f_t)\Big\| \\
& \le & \frac{1}{h}\int_t^{t+h} \Big\|Q(f_s,f_s)-Q(f_t,f_t)\Big\|\,ds.
 \end{eqnarray*}
Thanks to  \eqref{BoundBoltz3} we infer that
 $$  \Big\|\frac{f_{t+h}-f_t}{h}-Q(f_t,f_t)\Big\|\le 8\|f_0\|\frac{1}{h}\int_t^{t+h} \|f_s-f_t\|\,ds, $$
which goes to 0 as $h\to 0$, since $f$ is continuous. Therefore,  \eqref{BoltzDebil}
can be rewritten as
$$ \p_tf = Q(f,f).$$

\end{rem}

\medskip

We complete the proof of Theorem~\ref{ThmExistenceBoltzmann}  showing the uniqueness
and non-negativity of $f_t$.

\begin{prop} \label{ProofExistenciaBoltzLemmaPositivity}
Let $g_0\in P(K)$ and $\lambda\ge1$. There exists a unique $g \in$
$C^1([0,+\infty),M_+(K))$ such that $g_{|t=0}=g_0$ and for $t>0$ solving
\begin{equation}\label{BoltzExistencePositivityGammaPositif10}
 \p_tg_t + \lambda g_t  = Q(g_t,g_t) + \lambda g_t \int_K dg_t.
\end{equation}

\end{prop}

\begin{rem}  Notice that
$\|f_t\|=1$ guarantees that $f_t$ is a solution to
$(\ref{BoltzExistencePositivityGammaPositif10})$. By uniqueness $f_t$ must belong to
$M_+(K)$, hence is nonnegative.
\end{rem}

\begin{proof}
We begin introducing some definitions. Let  $\Gamma:M(K)\times M(K)\to M(K)$ be a
measure determined  by $\Gamma(f,g)=Q(f,g)+\frac{\lambda}{2}(g\int f+f\int g)$. Denote
$\Gamma(f):=\Gamma(f,f)$ and $Q(f)=Q(f,f)$. In view of \eqref{BoundBoltz3},
$\Gamma(f)$ is continuous in $f$ with respect to the total variation norm.

Moreover, we claim that  $\Gamma(f,g)\ge 0$ if $f$ and $g$ are non-negative. To see
this, note that the measure $Q$ can be represented by $Q(f,g) = Q_+(f,g) - Q_-(f,g)$
with
\begin{equation}\label{ProofExistenciaBoltzDefQ+}
 (Q_+(f,g),\phi)
=  \frac12 \int_{K^2\times B}
\phi(\varpi')\,\Big(df(\varpi)dg(\varpi_*)+dg(\varpi)df(\varpi_*)\Big)d\theta(\eta)
\end{equation}
and
\begin{eqnarray}\label{ProofExistenciaBoltzDefQ-}
 Q_-(f,g) = \frac12  \left(f\int_{K\times B} dg(\varpi)d\theta(\eta)+g\int_{K\times B} df(\varpi)d\theta(\eta)\right),
\end{eqnarray}
for any $\phi\in C(K)$.

Then, $\Gamma$ can be  expressed as
$$ \Gamma(f,g) = Q_+(f,g) + \frac12 (\lambda-1) \left(f\int_{K\times B} dg(\varpi)d\theta(\eta)+g\int_{K\times B} df(\varpi)d\theta(\eta)\right)\ge 0, $$
because $\lambda\ge 1$ and the claim follows.

Furthermore, whenever $g\ge f\ge 0$,
\begin{equation}\label{BoltzExistencePositivityGammaPositif}
\Gamma(g)\ge \Gamma(f)\ge 0.
\end{equation}
Indeed, since  $g+f$ and $g-f$ are non-negative measures,
$$ \Gamma(g)-\Gamma(f)=\Gamma(g+f,g-f)\ge 0. $$

We need to find $g\in  C([0,+\infty,M_+(K))$ such that  $\int g\le 1$ and
\begin{equation}\label{BoltzExistencePositivity2}
 g_t = e^{-\mu t}g_0 + \int_0^t e^{-\mu(t-s)}\Gamma(g_s)\,ds.
\end{equation}

It will be obtained as the limit of the following sequence $g^n:[0,+\infty)\to M(K)$,
$n\ge 0$, defined iteratively by $g^0=0$ and
\begin{equation}\label{BoltzExistencePositivityGammaSequence}
 g^n_t = e^{-\lambda t}g_0 + \int_0^t e^{-\lambda(t-s)}\Gamma(g^{n-1}_s)\,ds.
\end{equation}
Since $g_0\ge 0$ and the measure $\Gamma$ is continuous, non-negative and
non-decreasing, it is easy to see that  $g^n_t\ge g^{n-1}_t\ge 0$ for any $n$ and
$t>0$. Clearly, $g^n\in C([0,+\infty),M(K))$. Even more, by integrating  equation
\eqref{BoltzExistencePositivity2} in $K$,  taking into account that $\int Q(f)=0$ for
any $f\in M(K)$, we deduce that the total mass $g_t^n(K)=\int_K dg^n_t$ satisfies
$$ g_t^n(K)=e^{-\lambda t} +\lambda \int_0^t e^{-\lambda(t-s)}\Big(g^{n-1}_s(K)\Big)^2\,ds. $$
By induction,  $g_t^n(K)\le 1$  for any $t$.
%Moreover $m^n$ is non-decreasing in $n$. Thus there exists the limit
%$m(t):=\lim_{n\to +\infty} m^n(t) \in [0,1]$.
Therefore, for any non-negative $\phi\in C(K)$, the sequence $(\int_K\phi\, dg^n_t)_n$
is non-decreasing and bounded. It ensures the existence of a limit
$(g_t,\phi):=\lim_{n\to\infty} \int \phi\, dg^n_t$.

The estimate  \eqref{BoundBoltz2} yields that $\|\Gamma(g^n_t)\|\le 2+\lambda$
uniformly in $n,t$. Then for any $T>0$, it  follows then that
$$  \|g^n_t-g^n_s\|\le C(T)|s-t|,$$
for any $s,t\in [0,T]$ and any $n$. Applying Arzela-Ascoli theorem, we have, up to a
subsequence, that $g^n\to g$ in $C_{loc}([0,+\infty),M^+(K))$, which implies that $g_t(K)\leq 1$ and
 $g\in C_{loc}([0,+\infty),M_+(K))$. Passing to the limit in \eqref{BoltzExistencePositivityGammaSequence},
we get  that $g$ satisfies \eqref{BoltzExistencePositivity2}.

Observe that the continuity of $\Gamma$ guarantees that $g$ belongs  in fact to $C^1$.

Eventually, if $\tilde g$ is another solution of  \eqref{BoltzExistencePositivity2},
then, by  \eqref{BoundBoltz3},
$$ \|g_t-\tilde g_t\|\le \int_0^t  e^{-\mu(t-s)}\|\Gamma(g_s)-\Gamma(\tilde g_s)\|\,ds
 \le C\int_0^t  e^{-\mu(t-s)} \|g_s-\tilde g_s\|\,ds, $$
so that $\|g_t-\tilde g_t\|=0$ by Gronwall's Lemma. As a result $g=\tilde g$ and
the proof is finished.
\end{proof}

\section{Grazing Limit}

We perform exhaustively the passage to the grazing limit, considering all of the possible balances between the transport and the diffusion terms. Our proof is based on the arguments given in  \cite{T}, adapted  to our specific model.

\begin{thm} \label{ThmGrazingLimit}
 In the interaction rule \eqref{Regla} admit that $\sigma^2=\gamma\lambda$ for some $\lambda>0$.
Given an initial condition $f_0 \in P(K)$, consider the solution, $f$, of the
Boltzmann-like equation \eqref{BoltzDebil2} given by Theorem
\ref{ThmExistenceBoltzmann}. If $f_\gamma(\tau):=f(t)$,  stands for the time rescaled
probability density according to $\tau=\gamma t$, it holds, up to subsequences, that
$f_\gamma\to g$ as $\gamma\to 0$ in $C([0,T],P(K))$ for any $T>0$. Furthermore, the
limit $g\in C([0,+\infty),P(K))$   satisfies
 for any $\tau\ge 0$ and any $\phi\in C^\infty(K)$,
\begin{equation}\label{GrazingLimitEqu}
\begin{split}
  \int_K  \phi \,dg_\tau
    = & \int_K \phi \,df_0
   + \int_0^\tau \int_K (m(\tau)-w) \pmean  q\phi_w(\varpi) \, dg_s(\varpi) ds \\
      & + \frac{\lambda}{2}\int_0^\tau \int_Kq^2 D^2(|w|)\phi_{ww}(\varpi) \, dg_s(\varpi) ds.
\end{split}
\end{equation}
Moreover, if $\frac{\sigma^2}{\gamma}=\lambda\to 0$, then  $g\in C([0,+\infty),P(K))$ verifies the transport equation
\begin{equation}\label{FPSinRuidoDebil}   \int_K  \phi \,dg_\tau = \int_K \phi \,df_0
    + \int_0^\tau \int_K (m(\tau)-w)\pmean q\phi_w(\varpi) \, dg_s(\varpi) ds.
    \end{equation}
    If conversely, $\frac{\sigma^2}{\gamma}\to
+\infty$,   rescaling time as
$\tau:=\gamma^\alpha t$, for some $\alpha \in (0,1)$, it holds that $f_\gamma\to g$ as $\gamma\to 0$, where $g$ is determined by \begin{equation}\label{dif}\int_K  \phi \,dg_\tau = \int_K  \phi \,df_0
+ \frac{\lambda}{2}\int_0^\tau\int_{K^2} \phi_{ww}(\varpi) q^2
D^2(|w|)\,dg_\tau(\varpi),
\end{equation}
being $\lambda>0$ now such that $\sigma^2=\lambda\gamma^\alpha$.

\end{thm}

\begin{proof}
First of all consider the case $\sigma^2=\gamma\lambda$ for some $\lambda>0$. Let  $\phi\in C^3(K)$. The rescaled measure  $f_{\gamma}(\tau)$  solves
$$
\begin{array}{ll}
 \D\frac{d}{d\tau} \int_K  \phi(\varpi)f_{\gamma}(\tau,\varpi) d\varpi
  \\=\D\frac1{\gamma}\int_B \int_{K^2}
  ( \phi(\varpi')-\phi(\varpi)) f_{\gamma}(\tau,\varpi) f_{\gamma}(\tau,\varpi_*) d\varpi d\varpi_*d\theta(\eta).
\end{array}
$$
Recall that $\varpi=(w,p,q)$ and $\varpi'=(w',p,q)$. We perform a Taylor expansion of
$\phi$ (with respect to the $w$ variable) up to second order:
$$
\phi(\varpi')-\phi(\varpi)=\phi_w(\varpi)(w'-w)+\frac12\phi_{ww}(\tilde
\varpi)(w'-w)^2,
$$
with $\tilde \varpi=(\tilde w,p,q)$ being  $\tilde w=\theta w+(1-\theta)w'$ for some
$\theta\in (0,1)$. Note that $\int_B\eta\Theta(\eta)d\eta=0$ and
$\int_B\eta^2\Theta(\eta)d\eta=\sigma^2$. Then substituting this expansion into the
previous equation and using the updating rules  \eqref{Regla}, it yields
\begin{equation} \label{Grazing1}
\begin{split}
 & \frac{d}{d\tau} \int_K  \phi(\varpi)f_\gamma(\tau,\varpi) d\varpi  \\
  & =\D\int_{K^2}\phi_w(\varpi) p_*q(w_*-w) f_{\gamma}(\tau,\varpi) f_{\gamma}(\tau,\varpi_*) d\varpi d\varpi_*\\
  & \hspace{0.3cm} +\D\frac{\sigma^2}{2\gamma}\int_{K}\phi_{ww}(\varpi) q^2 D^2(|w|)f_{\gamma}(\tau,\varpi)  d\varpi \\
 & \hspace{0.3cm}  +\D\frac{\gamma}{2}\int_{K^2}\phi_{ww}(\varpi)  p_*^2q^2(w_*-w)^2 f_{\gamma}(\tau,\varpi) f_{\gamma}(\tau,\varpi_*) d\varpi d\varpi_* +R(\tau,\gamma,\sigma),
\end{split}
\end{equation}
where
$$\begin{array}{l}
\D R(\tau,\gamma,\sigma)=\frac{1}{2\gamma}\int_B \int_{K^2}\Theta(\eta)\Big(\gamma p_*q(w_*-w)+\eta qD(|w|)\Big)^2\\[0.4cm]\hspace{4cm}\D(\phi_{ww}(\tilde \varpi)-\phi_{ww}(\varpi))f_{\gamma}(\tau,\varpi)f_{\gamma}(\tau,\varpi_*)d\varpi d\varpi_*d\eta.
\end{array}$$

Observe that the first integral in \eqref{Grazing1} can be written as follows
$$\begin{array}{ll}
I&=\D \int_{K} q\phi_w(\varpi)f_{\gamma}(\tau,\varpi)d\varpi
         \int_K p_*w_*f_{\gamma}(\tau,\varpi_*) d\varpi_*\\
&\hspace{0.4cm} \D -\int_{K}qw\phi_w(\varpi)f_{\gamma}(\tau,\varpi)d\varpi
                             \int_{K} p_*f_{\gamma}(\tau,\varpi_*) d\varpi_*\\
&=\D \int_{K}( \langle wp\rangle-\pmean w)q\phi_w(\varpi)f_{\gamma}(\tau,\varpi)d\varpi \\
&\D=\int_{K}(m_\gamma(\tau)-w)\pmean q\phi_w(\varpi)f_{\gamma}(\tau,\varpi)d\varpi,
\end{array}$$
where $m_\gamma(\tau):=\frac1{\pmean }\int_K wpf_{\gamma}(\tau,\varpi)d\varpi$. The
previous analysis implies that
$$
\begin{array}{ll}
 \D\frac{d}{d\tau} \int_K  \phi(\varpi)f_{\gamma}(\tau,\varpi) d\varpi
  \\=\D\int_{K}(m_\gamma(\tau)-w)\pmean  q\phi_w(\varpi)f_{\gamma}(\tau,\varpi)d\varpi \\
    \hspace{0.4cm}+\D \frac{\sigma^2}{2\gamma}\int_{K}q^2 D^2(|w|)\phi_{ww}(\varpi)f_{\gamma}(\tau,\varpi)d\varpi \\
  \hspace{0.4cm}+\D\frac{\gamma}{2}\int_{K^2}\phi_{ww}(\varpi)  p_*^2q^2(w_*-w)^2 f_{\gamma}(\tau,\varpi) f_{\gamma}(\tau,\varpi_*) d\varpi d\varpi_* +R(\tau,\gamma,\sigma),
\end{array}
$$
which integrated in time gives,
 \begin{equation}\label{gorda}
\begin{array}{ll}
 \D \int_K  \phi(\varpi)(f_{\gamma}(\tau',\varpi)-f_{\gamma}(\tau,\varpi)) d\varpi   \\
  =\D\int_{\tau}^{\tau'}\int_{K}(m_\gamma(\tau)-w)\pmean q\phi_w(\varpi)f_{\gamma}(s,\varpi)d\varpi ds\\
    \hspace{0.4cm} \D +\frac{\sigma^2}{2\gamma}\int_{\tau}^{\tau'}\int_{K}q^2 D^2(|w|)\phi_{ww}(\varpi)f_{\gamma}(s,\varpi)d\varpi ds \\
 \hspace{0.4cm} +\D\frac{\gamma}{2}\int_{\tau}^{\tau'}\int_{K^2}\phi_{ww}(\varpi)  p_*^2q^2(w_*-w)^2 f_{\gamma}(s,\varpi) f_{\gamma}(s,\varpi_*) d\varpi d\varpi_* ds \\
\hspace{0.4cm} \D+\int_{\tau}^{\tau'}R(s,\gamma,\sigma)ds.
\end{array}
\end{equation}

We now show that, whenever $\sigma^2/\gamma$ remains bounded as $\gamma,\sigma\to 0$,
then
\begin{equation} \label{GrazingRemainder}
 \lim_{\gamma,\sigma\to 0} R(\tau,\gamma,\sigma)\to 0 \qquad \text{uniformly in $\tau\in \R$. }
\end{equation}
Using that $\varpi=(w,p,q)$, $\varpi'=(w',p,q)$ and $|\tilde w-w|=(1-\theta)|w'-w|\leq
|w'-w|$, we easily see that
$$
|\phi_{ww}(\tilde \varpi)-\phi_{ww}(\varpi)|\leq\|\phi_{www}\|_\infty|\tilde w-w|\leq
\|\phi_{www}\|_\infty|w'-w|.
$$
As a result,
\begin{align*}
|R(\tau,\gamma,\sigma)|& \leq \D
\frac{\|\phi_{www}\|_\infty}{2\gamma}\int_{B}\int_{K^2}\Theta(\eta)\Big|\gamma
p_*q(w_*-w) \\
 & \quad +\eta
qD(|w|)\Big|^3f_{\gamma}(\tau,\varpi)f_{\gamma}(\tau,\varpi_*)d\varpi d\varpi_*d\eta.
\end{align*}
Applying the inequality $(a+b)^3\leq 8(a/2+b/2)^3\leq 4(a^3+b^3)$ and taking into
account that $p_*,q,\gamma,D(|w|)\in [0,1]$ and $w,w_*\in[-1,1]$, we deduce that
$$
 \left|\gamma p_*q(w_*-w)+\eta qD(|w|)\right|^3\leq 4( |\gamma p_*q(w_*-w)|^3+|\eta qD(|w|)|^3)\leq 32\gamma^3+4\eta^3.
$$
Consequently,
\begin{equation}\label{GrazingCotaR}
 |R(\tau,\gamma,\sigma)|
=\|\phi_{www}\|_\infty\left(16\gamma^2+\frac{2\sigma^2}{\gamma}\sigma E[|Y|^3]\right).
\end{equation}
and the limit \eqref{GrazingRemainder}  follows since we assumed $E[|Y|^3]<\infty$.

We denote $X=C^3(K)$ with the usual norm $\|\phi\|_X=\sum_{|\alpha|\le 3}
\|\p^\alpha\phi\|_\infty$. Recall that $p,q,p^*,q^*, D(|w|)\in[0,1]$,
$f_{\gamma}(\tau, \cdot)\in P([-1,1])$ for all $\tau$ and $m_\gamma(t), w, w^*\in
[-1,1]$. Invoking \eqref{gorda} and \eqref{GrazingCotaR} it can be inferred that

\begin{equation*}
\begin{split}
& \left|\int_K  \phi(\varpi)(f_{\gamma}(\tau,\varpi)-f_{\gamma}(\tau',\varpi)) d\varpi\right|  \\
&\qquad \leq \D\left[2\|\phi_w\|_\infty+\|\phi_{ww}\|_\infty\left(2\gamma+\frac{\sigma^2}{\gamma}\right)\right. \\
&\qquad \qquad \left. +\|\phi_{www}\|_\infty\left(16\gamma^2+\frac{2\sigma^2}{\gamma}\sigma E[|Y|^3]\right)\right](\tau'-\tau)\\
&\qquad  \leq \|\phi\|_{X}\left(2+2\gamma+16\gamma^2+\frac{2\sigma^2}{\gamma}\sigma E[|Y|^3]+\frac{\sigma^2}{\gamma}\right)(\tau'-\tau). \\
&\qquad = C \|\phi\|_X (\tau'-\tau).
\end{split}
\end{equation*}
Taking supremum gives
\begin{equation}\label{GrazingEquCont}
\sup_{\phi\in X,\|\phi\|_{X}\leq 1}\left| \int_K
\phi(\varpi)(f_{\gamma}(\tau,\varpi)-f_{\gamma}(\tau',\varpi)) d\varpi\right|\leq
C(\tau'-\tau).
 \end{equation}
Define
\begin{equation}\label{GrazingDefNorm}
 \|\mu\| :=\sup_{\phi\in X,\|\phi\|_{X}\leq 1}\int_K\phi\,d\mu.
\end{equation}
Then, \eqref{GrazingEquCont} can be read as
$$ \|f_\gamma( \tau)-f_\gamma(\tau')\|\le C|\tau'-\tau|, $$
for any $\gamma\in [0,1]$ and any $\tau,\tau'\in [0,+\infty)$, where the constant $C$
is independent of $\gamma,\tau,\tau'$. It can be shown  that the norm in
\eqref{GrazingDefNorm}  induces  the weak topology on $P(K)$ (see Ref.[17], Lemma 5.3
and, Corollary 5.5). We have thus shown that the sequence of continuous probability
measure valued functions $f_\gamma:[0,+\infty)\to P(K)$ are uniformly equicontinuous.
In addition, $\|f_\gamma(\tau)\| \leq 1$ for any $\tau$ and $\gamma$, hence
Arzela-Ascoli  theorem, together with a diagonal argument, ensure the existence of
$g\in C([0,\infty);P(K))$ and a subsequence $(\gamma_n)_n$ converging to $0$ such that
$f_{\gamma_n} \to g$ in $C([0,T];P(K))$ for any $T>0$.

It remains to pass to the limit  in \eqref{gorda}. Since the norm in
\eqref{GrazingDefNorm} metrizes the weak convergence, it is well known (see for example \cite{V}) that
$$ \max_{\tau\in[0,T]}\|f_{\gamma_n}(\tau)-g(\tau)\| \to 0\quad \mbox{ as } n\to \infty, $$
can be expressed in terms of the Wasserstein distance as
\begin{equation}\label{disto}
\lim_{n\to\infty}\max_{\tau\in[0,T]}W_1(f_{\gamma_n}(\tau),g(\tau))= 0.
\end{equation}
We can rewrite  (\ref{disto}) as
\begin{equation}\label{forlip}
\int_K\varphi(\varpi)f_{\gamma_n}(\tau,\varpi)d\varpi \to
\int_K\varphi(\varpi)g(\tau,\varpi)d\varpi,
\end{equation}
uniformly on compacts $0\leq \tau\le T$ for any $T>0$ and for any Lipschitz function
$\varphi$. As a result we have
$$ m_\gamma(\tau) = \frac 1{\pmean}\int_K wp f_\gamma(\tau,\varpi) d\varpi
 \to \frac 1{\pmean}\int_K wp  g(\tau,\varpi) d\varpi =: m(\tau),  $$
uniformly for $\tau \in [0,T]$, $T>0$. Passing to the limit  in \eqref{gorda} this shows  that for any $\phi\in C^3(K)$
and any $\tau'\ge \tau\ge 0$,
\begin{equation*}
\begin{split}
  \int_K  \phi \,dg_{\tau'}
    = & \int_K \phi \,dg_\tau
    + \int_\tau^{\tau'} \int_K (m(\tau)-w) \pmean q\phi_w(\varpi) \, dg_s(\varpi) ds \\
      & + \frac{\lambda}{2}\int_\tau^{\tau'}\int_Kq^2 D^2(|w|)\phi_{ww}(\varpi) \, dg_s(\varpi) ds,
\end{split}
\end{equation*}
which proves  \eqref{GrazingLimitEqu} as desired.

Admit now that $\frac{\sigma^2}{\gamma}\to 0$. Taking limit as $\lambda\to0$ in
\eqref{gorda} shows \eqref{FPSinRuidoDebil}.

Finally, suppose that $\sigma^2=\lambda\gamma^\alpha$ for some $\lambda>0$ and
$\alpha\in (0,1)$. In particular $\frac{\sigma^2}{\gamma}\to +\infty$ as $\gamma\to
0$, hence the diffusion dominates the transport. Rescaling time as
$\tau:=\gamma^\alpha t$, \eqref{Grazing1} now reads as
\begin{equation*}
\begin{split}
 & \frac{d}{d\tau} \int_K  \phi(\varpi)f_\gamma(\tau,\varpi) d\varpi  \\
  & =\D \gamma^{1-\alpha} \int_{K^2}\phi_w(\varpi) p_*q(w_*-w) f_{\gamma}(\tau,\varpi) f_{\gamma}(\tau,\varpi_*) d\varpi d\varpi_*\\
  & \hspace{0.3cm} +\D\frac{ \lambda}{2}\int_{K}\phi_{ww}(\varpi) q^2 D^2(|w|)f_{\gamma}(\tau,\varpi)  d\varpi \\
 & \hspace{0.3cm}  +\D\frac{\gamma^{2-\alpha}}{2}\int_{K^2}\phi_{ww}(\varpi)  p_*^2q^2(w_*-w)^2 f_{\gamma}(\tau,\varpi) f_{\gamma}(\tau,\varpi_*) d\varpi d\varpi_* +\tilde R(\tau,\gamma,\sigma),
\end{split}
\end{equation*}
where $\tilde R(\tau,\gamma,\sigma) = \gamma^{3-\alpha}R(\tau,\gamma,\sigma)$. Using
\eqref{GrazingCotaR}, we have
$$|\tilde R(\tau,\gamma,\sigma)|\le \|\phi\|_X (16\gamma^{5-\alpha} + 2\lambda\gamma^2\sigma E|Y|^3)
= o(1)\|\phi\|_X.$$ Arguing as before it can be shown that the limit $g$ satisfies \eqref{dif}, and the proof is complete.

\end{proof}

\section*{Acknowledgment}

This work was partially supported by Universidad de Buenos Aires under grants
20020150100154BA and
 20020130100283BA,   by ANPCyT PICT2012 0153 and PICT2014-1771,   CONICET (Argentina) PIP 11220150100032CO and 5478/1438. All the authors are members of CONICET, Argentina.
 MPL and NS wish to thank the members of UNSL-IMASL, for their hospitality and support.

%\section*{References}


\begin{thebibliography}{00}



\bibitem{ANT} G. Aletti, G. Naldi \& G. Toscani
{\it First-order continuous models of opinion formation},   SIAM J. Appl. Math., {\bf 67} (3), (2007), 837--853.

\bibitem{ABD} L. Arlotti, N. Bellomo \& E. De Angelis,
{\it Generalized kinetic Boltzmann models: Mathematical structures and applications},
 Math. Mod. Meth. in Appl. Sci., {\bf 12}, (2002), 571--596.


\bibitem{Ash} R.B. Ash, {\it Real Analysis and Probability}, Probability and Mathematical Statistics, (11), Academic Press, New York-London, (1972).


\bibitem{Bell} N. Bellomo, {\it Modeling Complex Living Systems A Kinetic Theory and Stochastic Game Approach}. Birkh\"auser, (2008).

\bibitem{Bell2} N. Bellomo, G. Ajmone Marsan \& A. Tosin, {\it Complex Systems and Society. Modeling
and Simulation.} Springer Briefs in Mathematics, (2013).

\bibitem{BenNaim1} E. Ben-Naim, P. L. Krapivsky, \& S. Redner,{\it
Bifurcations and patterns in compromise processes}. Phys. D,  {\bf 183} 3-4 (2003),  190--204.


\bibitem{BenNaim2} E. Ben-Naim, P. L. Krapivsky,  F Vazquez,  \& S. Redner.
 {\it  Unity and discord in opinion dynamics}. Phys. A,  {\bf 330} 1-2 (2003), 99--106.

\bibitem{BGV} F. Bolley, A. Guillin \& C. Villani, {\it Quantitative
concentration inequalities for empirical measures on non-compact
spaces},  Probab. Theory Relat. Fields, {\bf 137}, (2007), 541--593.


\bibitem{BV} F. Bolley \& C. Villani, {\it Weighted  Csisz\'ar-Kullback-Pinsker
inequalities and applications to transportation inequalities},  Ann. Fac. Sci. Toulouse Math.,
{\bf 14}, 6 (3), (2005), 331--352.

\bibitem{BS} L. Boudin \& F. Salvarani, {\it A kinetic approach to the study
of opinion formation}, Math. Model. Numer. Anal., {\bf 43}, (2009), 507--522.


\bibitem{BrT} C. Brugna \& G. Toscani, {\it Kinetic models of opinion formation in the presence of personal conviction}, Phys. Rev. E, {\bf 92}, (2015).


\bibitem{CB} V. Capasso \& D. Bakstein, {\it An introduction to
continuous-time Stochastic Processes}, Modeling and Simulation in Science, Engineering and Technology. Birkh\"auser/Springer, New York, (2015).

\bibitem{CCC} B.K. Chakrabarti, A. Chakraborti \& A. Chatterjee,
{\it Econophysics and Sociophysics : Trends and Perspectives}, Wiley-VCH,
Berlin, (2006).

\bibitem{CCR} J.A. Ca\~nizo, J.A. Carrillo \&J. Rosado, A well-posedness
theory in measures for some kinetic models of collective motion, Math. Mod. Meth. in Appl. Sci., {\bf 21} (3), (2011), 515--539.

\bibitem{CI} C. Cercignani, R. Illner \& M. Pulvirenti, {\it The
Mathematical Theory of Dilute Gases}, Springer
Series in Applied Mathematical Sciences, Springer-Verlag, 106, (1994).

\bibitem{DNAW} G. Deffuant, D. Neau, F. Amblard \& G. Weisbuch, {\it Mixing
beliefs among interacting agents}, Advances in Complex Systems, {\bf 3}, (2000), 87--98.

\bibitem{DAWF} G. Deffuant, F. Amblard, G. Weisbuch \& T. Faure, {\it How can
extremism prevail? A study based on the relative agreement interaction
model}, Journal of Artificial Societies and Social Simulations, {\bf 5} (4), (2002).

\bibitem{DL} P. Degond \& B. Lucquin-Desreux, {\it The Fokker-Planck asymptotics of the Boltzmann collision
operator in the Coulomb case}, Math. Mod. Meth. in Appl. Sci., {\bf 2}, (1992), 167--182.

\bibitem{D} L. Desvillettes, {\it On asymptotics of the Boltzmann equation when the collisions become grazing},
Trans. Theo. Stat. Phys., {\bf 21}, (1992), 259--276.

\bibitem{DR} L. Desvillettes \& V. Ricci, {\it A rigorous derivative of a linear kinetic equation of
Fokker-Planck type in the limit of grazing collisions}, Journal of Statistical Physics, {\bf 104}, (2001), 1173--1189.

\bibitem{DiL} A. Di Mare \& V. Latora, {\it Opinion formation models based on
game theory}, International Journal of Modern Physics C, {\bf 18}, (2007), 1377--1395.

\bibitem{DM} B. D\"uring, P. Markowich, J. F. Pietschmann \& M-T. Wolfram,
{\it Boltzmann and Fokker-Planck equations modelling opinion formation in
the presence of strong leaders}, Proc. Royal Society A, {\bf 465}, (2009), 1--22.

\bibitem{DW}
B. D\" uring \& M. T. Wolfram, {\it  Opinion dynamics: inhomogeneous Boltzmann-type
equations modelling opinion leadership and political segregation}.
Proc. R. Soc. A  471 No 2181 (2015) 20150345.

\bibitem{EH} P. Embrechts \& M. Hofert, {\it A note on generalized inverse},
available at \begin{verbatim}
https://people.math.ethz.ch/~embrecht/ftp/generalized_inverse.pdf\end{verbatim}

\bibitem{GTW} G. Gabetta, G. Toscani \& B. Wennberg, {\it Metrics for
probability distributions and the trend to equilibrium for solutions of the Boltzmann equation},
Journal of Statistical Physics, {\bf 81} (5-6), (1995), 901--934.

\bibitem{G1} S. Galam, {\it Heterogeneous beliefs, segregation, and
extremism in the making of public opinions}, Phys. Rev. E, {\bf 71}, (2005).

\bibitem{G2} S. Galam, {\it Sociophysics: A Physicist's Modeling of
Psycho-political Phenomena}, New York, Springer-Verlag, (2012), 439 p.

\bibitem{GZ} S. Galam \& J.D. Zucker, {\it From individual choice to group
decision-making}, Phys.  A, {\bf 287}, (2000), 644--659.

\bibitem{Gal} S. Galam,   {\it Sociophysics: a physicist's modeling of psycho-political phenomena},
Springer Science \& Business Media, (2012).

\bibitem{autom}
J. Ghaderi \&  R. Srikant. {\it Opinion dynamics in social networks with stubborn agents: Equilibrium and convergence rate},  Automatica, {\bf 50} (12), (2014), 3209--3215.


\bibitem{G} M.I. Gil, {\it Norm estimations for operator-valued functions
and applications}, Monographs and Textbooks in Pure and Applied Mathematics, 192. Marcel Dekker, Inc., New York, (1995).

\bibitem{H} P.J. Huber, Robust Statistics, John Wiley \& Sons, Inc., Hoboken, NJ, USA (1981).

\bibitem{Kree} P. Kree \& C. Soize, {\it Mathematics of Random Phenomena: random vibrations of mechanical structures},
Mathematics and Its Applications, D. Reidel Publishing Co., (1986).

\bibitem{LBG} P.F.Lazarsfeld, B.R. Berelson \& H. Gaudet, {\it The people's
choice: how the voter makes up his mind in a presidential campaign}, New York, NY: Duell, Sloan and
Pierce, (1944).

\bibitem{LT} H. Li \& G. Toscani,  {\it Long-time asymptotics of kinetic models of granular flows}, Archive for Rational Mechanics and Analysis, {\bf 172}  (3), (2004), 407--428.

\bibitem{lipow} A. Lipowski, D. Lipowska \& A. L. Ferreira.  {\it Agreement dynamics on directed random graphs}. J. Stat. Mech. Theory Exp. (6), (2017).

\bibitem{SM} P. Milgrom \& I. Segal, {\it Envelope theorems for arbitrary choice sets},
Econometrica, {\bf 70} (2), (2002), 583--601.


\bibitem{Mo1}
M. Mobilia,  {\it Does a single zealot affect an infinite group of voters?}, Physical Review Letters, {\bf 91} (2), (2003).

\bibitem{Mob} M. Mobilia,  A. Petersen \&  S. Redner. {\it On the role of zealotry in the voter model},
  J. Stat. Mech. Theory Exp., (8),   (2007).


\bibitem{NPT} G. Naldi, L. Pareschi \& G. Toscani  {\it  Mathematical modeling of collective behavior in socio-economic and life sciences}.
Modeling and Simulation in Science, Engineering and Technology. Birkh\"auser Boston, Inc., Boston, MA, (2010).



\bibitem{PT} L. Pareschi \& G. Toscani, {\it Interacting Multiagent Systems:
Kinetic Equations and Monte Carlo Methods}, Oxford University Press,
Oxford (2014).

\bibitem{SC} P.P. Sen \& B.K. Chakrabarti, {\it Sociophysics: An Introduction}, Oxford University Press, (2014).

\bibitem{Sla} F. Slanina, {\it Essentials of Econophysics Modelling}, Oxford University Press, (2014).

\bibitem{SL} F. Slanina \& H. Lavicka, {\it Analytical results for the Sznajd model of opinion formation},
Eur. Phys. J. B, {\bf 35}, (2003), 279--288.

\bibitem{S} K. Sznajd-Weron \& J. Sznajd, {\it Opinion evolution in closed
community}, Int. J. Mod. Phys. C, {\bf 11}, (2000), 157--1165.

\bibitem{Stroock} D.W. Stroock, {\it Probability theory, an analytic view}, Cambridge University Press, (1993).

\bibitem{T} G. Toscani, {\it Kinetic models of opinion formation}, Comm. Math. Sci., {\bf  4}, (3), (2006),  481--496.

\bibitem{Vaz} C. La Rocca, L.  A. Braunstein \&  F. V\'azquez, {\it
The influence of persuasion in opinion formation and polarization},  Europhys. Letters, {\bf 106}, (2014).



\bibitem{vil1} C. Villani, {\it A review of mathematical topics in collisional kinetic theory},
 Handbook of Mathematical Fluid Dynamics 1,  S. J. Friedlander \& D. Serre. Elsevier (2002).

\bibitem{V} C. Villani, {\it Topics in optimal transportation}, Grad.Studies in Math. (58),
American Mathematical Soc., (2003).

\bibitem{vil2} C. Villani, {\it Mathematics of granular materials},  J. Stat. Phys., {\bf 124}, (2006), 781--822.

\bibitem{V2} C. Villani, {\it On a new class of weak solutions for the spatially homogeneous Boltzmann and Landau
equations}, Arch. Rat. Mech. Anal., {\bf 143}, (1998), 273--307.

\bibitem{V3} C.Villani, {\it On the spatially homogeneous Landau equation for Maxwellian molecules}, Math. Mod.
Meth. in Appl. Sci., {\bf 8}, (1998), 957--983.

\bibitem{WVC}  A. Waagen, G. Verma, K. Chan, A. Swami \& R. D'Souza,
 {\it Effect of zealotry in high-dimensional opinion dynamics models}, Phys. Rev. E,  {\bf 91}, (2014).

\bibitem{WVaz}
A. Woolcock, C. Connaughton, Y. Merali, \& F. Vazquez,
{\it
Fitness voter model: Damped oscillations and anomalous consensus},
Phys. Rev. E  {\bf 96}, (2017) 032313.


\bibitem{xie}
J. Xie,  S. Sreenivasan, G. Korniss, W. Zhang, C. Lim \& B. K. Szymanski.
{\it Social consensus through the influence of committed minorities}, Phys. Rev.  E, {\bf 84}
(1), (2011).

\bibitem{XEK} J. Xie, J. Emenheiser, M. Kirby, S. Sreenivasan, B. K. Szymanski \&
G. Korniss. {\it  Evolution of Opinions on Social Networks in the Presence of Competing Committed},
Groups PLoS ONE, {\bf 7} (3), 2012.

\bibitem{Ozg}
E. Yildiz,  A. Ozdaglar, D. Acemoglu, A. Saberi \& A. Scaglione.
{\it Binary opinion dynamics with stubborn agents}. ACM Transactions on Economics and Computation,  {\bf 1}, (4), (2013).


\end{thebibliography}
\end{document}